\documentclass[dvipdfmx]{article}
\usepackage{amsthm, amscd, amsmath, amsfonts, mathrsfs}
\newtheorem{theo}{Theorem}[section]
\newtheorem{proposition}[theo]{Proposition}
\newtheorem{lemma}[theo]{Lemma}

\newtheorem{rem}[theo]{Remark}
\newtheorem{definition}[theo]{Definition}

\newtheorem{problem}[theo]{Problem}
\input xy
\xyoption{all}

\begin{document}

\title{Geometric Interpretation for Exact Triangles Consisting of Projectively Flat Bundles on Higher Dimensional Complex Tori}

\author{Kazushi Kobayashi\footnote{Department of Mathematics and Informatics, Graduate School of Science, Chiba University, Yayoicho 1-33, Inage, Chiba, 263-8522, Japan. E-mail : kazushi-kobayashi@chiba-u.jp. 2020 Mathematics Subject Classification : 14F08 (primary), 14J33 (secondary). Keywords : torus, projectively flat bundle, homological mirror symmetry.}}

\date{}

\maketitle

\begin{abstract}
Let $(X^n, \check{X}^n)$ be a mirror pair of an $n$-dimensional complex torus $X^n$ and its mirror partner $\check{X}^n$. Then, a simple projectively flat bundle $E(L,\mathcal{L})\rightarrow X^n$ is constructed from each affine Lagrangian submanifold $L$ in $\check{X}^n$ with a unitary local system $\mathcal{L}\rightarrow L$. In this paper, we first interpret these simple projectively flat bundles $E(L,\mathcal{L})$ in the language of factors of automorphy. Furthermore, we give a geometric interpretation for exact triangles consisting of three simple projectively flat bundles $E(L,\mathcal{L})$ and their shifts by focusing on the dimension of intersections of the corresponding affine Lagrangian submanifolds $L$. Finally, as an application of this geometric interpretation, we discuss whether such an exact triangle on $X^n$ ($n\geq 2$) is obtained as the pullback of an exact triangle on $X^1$ by a suitable holomorphic projection $X^n\rightarrow X^1$.
\end{abstract}

\tableofcontents

\section{Introduction}
In this paper, we construct a mirror pair of tori as an analogue of the SYZ construction \cite{SYZ}, and study exact triangles which appear in the discussions in the homological mirror symmetry \cite{Kon} for tori. The SYZ construction is conjectured by Strominger, Yau, and Zaslow in 1996, and it proposes a way of constructing mirror pairs geometrically. Roughly speaking, this construction is the following. A mirror pair of Calabi-Yau manifolds $(M,\check{M})$ is realized as the special Lagrangian torus fibrations $\pi : M\rightarrow B$ and $\check{\pi } : \check{M}\rightarrow B$ on the same base space $B$. In particular, for each point $b\in B$, the special Lagrangian torus fibers $\pi ^{-1}(b)$ and $\check{\pi }^{-1}(b)$ are related by the T-duality. On the other hand, the homological mirror symmetry is conjectured by Kontsevich in 1994, and it states the following. For each Calabi-Yau manifold $M$, there exists a Calabi-Yau manifold $\check{M}$ such that there exists an equivalence
\begin{equation*}
D^b(Coh(M))\cong Tr(Fuk(\check{M}))
\end{equation*}
as triangulated categories. Here, $D^b(Coh(M))$ is the bounded derived category of coherent sheaves on $M$, and $Tr(Fuk(\check{M}))$ is the derived category of the Fukaya category $Fuk(\check{M})$ on $\check{M}$ \cite{Fukaya category} obtained by the Bondal-Kapranov-Kontsevich construction \cite{bondal}, \cite{Kon}. One of the most fundamental examples of mirror pairs is a pair $(X^n, \check{X}^n)$ of tori, where $X^n$ is an $n$-dimensional complex torus and $\check{X}^n$ is a mirror partner of $X^n$, so there are many studies of the homological mirror symmetry for tori. For example, Polishchuk and Zaslow discuss the homological mirror symmetry in the case of elliptic curves, i.e., $(X^1, \check{X}^1)$ in \cite{elliptic} (the details of higher $A_{\infty }$-product structures are studied in \cite{A-inf}), and Fukaya studied the homological mirror symmetry for abelian varieties via the SYZ construction in \cite{Fuk}. In particular, in \cite{Fuk}, he discussed the homological mirror symmetry by focusing on the cases that objects of the Fukaya category are restricted to affine Lagrangian submanifolds with unitary local systems in the symplectic geometry side, and then, the corresponding holomorphic vector bundles are projectively flat. On the other hand, projectively flat bundles are examples of Einstein-Hermitian vector bundles, and Einstein-Hermitian vector bundles relate closely to stable vector bundles via the Kobayashi-Hitchin correspondence \cite{koba}, \cite{K-H}. Hence, projectively flat bundles are also stable. Thus, projectively flat bundles play a fundamental role in the complex or algebraic geometry, including the homological mirror symmetry for tori. Let $(L,\mathcal{L})$ be an object of the Fukaya category $Fuk(\check{X}^n)$, where $L\cong T^n$ is an affine Lagrangian (multi) section of the trivial special Lagrangian torus fibration $\check{\pi} : \check{X}^n\rightarrow T^n$ and $\mathcal{L}\rightarrow L$ is a unitary local system along $L$. Each object $(L,\mathcal{L})$ corresponds to a simple projectively flat bundle $E(L,\mathcal{L})\rightarrow X^n$ via the homological mirror symmetry. Here, special Lagrangian torus fibers of $\check{\pi} : \check{X}^n\rightarrow T^n$ with unitary local systems along them correspond to skyscraper sheaves on $X^n$. We can also regard this correspondence as an analogue of the Fourier-Mukai transform \cite{leung}, \cite{A-P}. Hereafter, we call an affine Lagrangian (multi) section simply an affine Lagrangian submanifold. By the definition of projectively flat bundles, a holomorphic vector bundle $E$ is projectively flat if and only if the curvature form of $E$ is expressed locally as $\alpha \cdot I_E$, where $\alpha $ is a complex 2-form and $I_E$ is the identity endmorphism of $E$. Furthermore, the classification result of factors of automorphy of projectively flat bundles on complex tori is given in \cite{Hano}, \cite{matsu}, \cite{koba}, \cite{yang}. The purposes of this paper are to characterize holomorphic vector bundles $E(L,\mathcal{L})$ by using factors of automorphy of projectively flat bundles on $X^n$, and to study exact triangles consisting of three simple projectively flat bundles $E(L,\mathcal{L})$ and their shifts on a given higher dimensional complex torus $X^n$. 

We explain the body of this paper briefly. Roughly speaking, the body of this paper consists of the two parts which are described below.

The first part is devoted to the study of the projective flatness of $E(L,\mathcal{L})$. For each holomorphic vector bundle $E(L,\mathcal{L})$, we can check easily that the curvature form of $E(L,\mathcal{L})$ is expressed locally as $\alpha \cdot I_{E(L,\mathcal{L})}$, where $\alpha $ is a complex 2-form and $I_{E(L,\mathcal{L})}$ is the identity endmorphism of $E(L,\mathcal{L})$, so $E(L,\mathcal{L})$ is projectively flat. However, the expression of the transition functions of $E(L,\mathcal{L})$ differs from the expression of the factor of automorphy of the projectively flat bundle $\mathcal{E}(L,\mathcal{L})$ which should be isomorphic to $E(L,\mathcal{L})$, so interpreting holomorphic vector bundles $E(L,\mathcal{L})$ in the language of factors of automorphy is a non-trivial problem. Thus, we interpret $E(L,\mathcal{L})$ in the language of factors of automorphy by constructing an isomorphism $E(L,\mathcal{L}) \stackrel{\sim }{\rightarrow } \mathcal{E}(L,\mathcal{L})$ explicitly (Theorem \ref{theo3.6}). 

In the second part, we mainly focus on a higher dimensional complex torus $X^n$, and study exact triangles consisting of simple projectively flat bundles $E(L,\mathcal{L})$ and their shifts on $X^n$. In general, holomorphic vector bundles $E(L,\mathcal{L})$ forms a DG-category $DG_{X^n}$. We expect that this $DG_{X^n}$ generates the bounded derived category of coherent sheaves $D^b(Coh(X^n))$ in the sense of the Bondal-Kapranov-Kontsevich construction,
\begin{equation*}
Tr(DG_{X^n})\cong D^b(Coh(X^n)).
\end{equation*}
At least, it is known that it split generates $D^b(Coh(X^n))$ when $X^n$ is an abelian variety (cf. \cite{orlov}, \cite{abouzaid}). Concerning these facts, in this paper, we focus on the triangulated category $Tr(DG_{X^n})$ instead of $D^b(Coh(X^n))$, and consider an exact triangle
\begin{align}
\begin{CD}
\cdots &@>>> E(L_a, \mathcal{L}_a) @>>> C(\psi) @>>> E(L_b, \mathcal{L}_b) \\
&@>\psi\not=0>> E(L_a, \mathcal{L}_a)[1] @>>> \cdots \label{intro}
\end{CD}
\end{align}
in $Tr(DG_{X^n})$. Here, $C(\psi )$ denotes the mapping cone of a non-trivial morphism $\psi : E(L_b,\mathcal{L}_b)\rightarrow E(L_a,\mathcal{L}_a)[1]$. By the definition of the DG-category $DG_{X^n}$, the degrees of morphisms between holomorphic vector bundles $E(L,\mathcal{L})$ are equal to or larger than 0 in $DG_{X_n}$. This fact implies that each exact triangle consisting of projectively flat bundles and their shifts is always expressed as the exact triangle of the form (\ref{intro}). In order to explain the statement of the main result in this paper, we now recall the previous work \cite{D} briefly. In \cite{D}, we studied the exact triangle of the form (\ref{intro}) under the assumptions $\mathrm{rank}\hspace{0.5mm}E(L_a, \mathcal{L}_a)=1$ and the existence of a holomorphic vector bundle $E(L_c, \mathcal{L}_c)\in \mathrm{Ob}(DG_{X^n})$ such that $C(\psi)\cong E(L_c, \mathcal{L}_c)$. Then, \cite[Theorem 5.6]{D} states that the exact triangle (\ref{intro}) essentially comes from a one-dimensional complex torus, i.e., it is obtained as the pullback of an exact triangle consisting of three projectively flat bundles and their shifts on a one-dimensional complex torus $X^1$ by a suitable holomorphic projection $\pi : X^n \rightarrow X^1$ (Definition \ref{deftr}). In this paper, we discuss a generalization of \cite[Theorem 5.6]{D} to the case that $\mathrm{rank}\hspace{0.5mm}E(L_a, \mathcal{L}_a)$ is not necessarily 1 (Problem \ref{conj}). More precisely, we show that the exact triangle (\ref{intro}) essentially comes from a one-dimensional complex torus if $\mathrm{rank}\hspace{0.5mm}E(L_a, \mathcal{L}_a)$ and $\mathrm{rank}\hspace{0.5mm}E(L_b, \mathcal{L}_b)$ are relatively prime, i.e., $gcd(\mathrm{rank}\hspace{0.5mm}E(L_a, \mathcal{L}_a), \mathrm{rank}\hspace{0.5mm}E(L_b, \mathcal{L}_b))=1$ (Theorem \ref{mainth}). Furthermore, we also give an example of an exact triangle (\ref{intro}) that essentially does not come from a one-dimensional complex torus in the case $gcd(\mathrm{rank}\hspace{0.5mm}E(L_a, \mathcal{L}_a), \mathrm{rank}\hspace{0.5mm}E(L_b, \mathcal{L}_b))\not=1$ under the assumption that the homological mirror symmetry conjecture for $(X^n, \check{X}^n)$ holds true. 

This paper is organized as follows. In section 2, we explain relations between objects $(L,\mathcal{L})$ of the Fukaya category $Fuk(\check{X}^n)$ and holomorphic vector bundles $E(L,\mathcal{L})$. Furthermore, we construct the DG-category $DG_{X^n}$ consisting of those holomorphic vector bundles $E(L,\mathcal{L})$. In section 3, we investigate some properties of holomorphic vector bundles $E(L,\mathcal{L})$. More precisely, for each holomorphic vector bundle $E(L,\mathcal{L})$, we find the projectively flat bundle $\mathcal{E}(L,\mathcal{L})$ which should be isomorphic to $E(L,\mathcal{L})$, and construct an isomorphism $E(L,\mathcal{L}) \stackrel{\sim }{\rightarrow } \mathcal{E}(L,\mathcal{L})$ explicitly. This result is given in Theorem \ref{theo3.6}. In sections 4, 5, we focus on the exact triangle of the form (\ref{intro}) under the assumption $C(\psi)\cong E(L_c, \mathcal{L}_c)$ for a suitable holomorphic vector bundle $E(L_c, \mathcal{L}_c)\in \mathrm{Ob}(DG_{X^n})$. In section 4, as a geometric interpretation for the exact triangle (\ref{intro}) from the viewpoint of the homological mirror symmetry for $(X^n, \check{X}^n)$, we prove codim$(L_a\cap L_b)=1$. This result is given in Theorem \ref{theo4.1}, and it plays a key role in section 5. The purpose of section 5 is to extend \cite[Theorem 5.6]{D} to general settings. In subsection 5.1, we recall the previous result \cite[Theorem 5.6]{D}. In subsection 5.2, in order to generalize \cite[Theorem 5.6]{D} to the case that $\mathrm{rank}\hspace{0.5mm}E(L_a, \mathcal{L}_a)$ is not necessarily 1, we reformulate the problem by focusing on a class of autoequivalences on $Tr(DG_{X^n})$. This is presented in Problem \ref{conj}. In subsection 5.3, we give an answer for Problem \ref{conj}. In particular, in Theorem \ref{mainth}, we prove that Problem \ref{conj} can be solved affirmatively for a certain class of exact triangles in $Tr(DG_{X^n})$. This is the main theorem in this paper.

\section{Holomorphic vector bundles and affine Lagrangian submanifolds with unitary local systems}
In this section, we consider a mirror pair $(T^{2n}_{J=T},\check{T}^{2n}_{J=T})$ of an $n$-dimensional complex torus $T^{2n}_{J=T}$ and its mirror partner $\check{T}^{2n}_{J=T}$, and discuss relations between affine Lagrangian submanifolds in $\check{T}^{2n}_{J=T}$ with unitary local systems and the corresponding holomorphic vector bundles on $T^{2n}_{J=T}$. This is based on the SYZ construction (SYZ transform) \cite{SYZ} (see also \cite{leung}, \cite{A-P}). Furthermore, we define a DG-category consisting of such holomorphic vector bundles. 

First, we explain the complex geometry side. We define a complex torus $T^{2n}_{J=T}$ as follows. Let $T$ be a complex matrix of order $n$ such that $\mathrm{Im}T$ is positive definite. We denote by $t_{ij}$ the $(i,j)$ component of $T$. Let us consider the lattice $L$ in $\mathbb{C}^n$ generated by 
\begin{gather*}
\gamma _1:=(2\pi ,0,\cdots, 0)^t,\cdots, \gamma _n:=(0,\cdots, 0,2\pi )^t,\\
\gamma '_1:=(2\pi t_{11},\cdots, 2\pi t_{n1})^t,\cdots, \gamma '_n:=(2\pi t_{1n},\cdots, 2\pi t_{nn})^t,
\end{gather*}
and define
\begin{equation*}
T^{2n}_{J=T}:=\mathbb{C}^n/L=\mathbb{C}^n/2\pi (\mathbb{Z}^n \oplus T\mathbb{Z}^n).
\end{equation*}
Sometimes we regard the $n$-dimensional complex torus $T^{2n}_{J=T}$ as a $2n$-dimensional real torus $\mathbb{R}^{2n}/2\pi \mathbb{Z}^{2n}$. In this paper, we further assume that $T$ is a non-singular matrix. Actually, in our setting described below, the mirror partner of $T^{2n}_{J=T}$ does not exist if det$T=0$. However, we can avoid this problem and discuss the homological mirror symmetry even if det$T=0$ by modifying the definition of the mirror partner of $T^{2n}_{J=T}$ and a class of holomorphic vector bundles which we treat. This fact is discussed in \cite{kazushi2}. We fix an $\varepsilon >0$ small enough and let
\begin{flalign*}
&O^{l_1\cdots l_n}_{m_1\cdots m_n}:=\biggl\{ \left(\begin{array}{ccc}x\\y\end{array}\right)\in T^{2n}_{J=T} \ |\ \frac{2}{3}\pi (l_j-1)-\varepsilon <x_j<\frac{2}{3}\pi l_j+\varepsilon , &\\
&\hspace{36mm}\frac{2}{3}\pi (m_k-1)-\varepsilon <y_k<\frac{2}{3}\pi m_k+\varepsilon, \ j,k=1,\cdots, n \biggr\}&
\end{flalign*}
be subsets of $T^{2n}_{J=T}$, where $l_j,m_k=1,2,3$, 
\begin{equation*}
x:=(x_1,\cdots, x_n)^t, \ y:=(y_1,\cdots, y_n)^t,
\end{equation*}
and we identify $x_i\sim x_i+2\pi$, $y_i\sim y_i+2\pi$ for each $i=1,\cdots, n$. Sometimes we denote $O^{l_1\cdots (l_j =l) \cdots l_n}_{m_1\cdots (m_k=m) \cdots m_n}$ instead of $O^{l_1\cdots l_n}_{m_1\cdots m_n}$ in order to specify the values $l_j=l$, $m_k=m$. Then, $\{ O_{m_1\cdots m_n}^{l_1\cdots l_n} \}_{l_j, m_k=1,2,3}$ is an open cover of $T^{2n}_{J=T}$. We define the local coordinates of $O^{l_1\cdots l_n}_{m_1\cdots m_n}$ by 
\begin{equation*}
(x_1,\cdots,x_n,y_1,\cdots,y_n)^t\in \mathbb{R}^{2n}.
\end{equation*}
Furthermore, we locally express the complex coordinates $z:=(z_1,\cdots,z_n)^t$ of $T^{2n}_{J=T}$ by $z=x+Ty$.

Now, we define a class of holomorphic vector bundles 
\begin{equation*}
E_{(r,A,\mu ,\mathcal{U})} \rightarrow T^{2n}_{J=T}. 
\end{equation*}
We first construct it as a complex vector bundle, and then discuss when it becomes a holomorphic vector bundle in Proposition \ref{pro2.1}. However, since the notations of transition functions of $E_{(r,A,\mu ,\mathcal{U})}$ are complicated, before giving the strict definition of $E_{(r,A,\mu ,\mathcal{U})}$, we explain the idea of the construction of $E_{(r,A,\mu ,\mathcal{U})}$. We assume $r\in \mathbb{N}$, $A=(a_{ij}) \in M(n;\mathbb{Z})$, and $p=(p_1,\cdots, p_n)^t$, $q=(q_1,\cdots, q_n)^t \in \mathbb{R}^n$. By using these $p$, $q\in \mathbb{R}^n$, we further define $\mu =(\mu _1 ,\cdots, \mu _n )^t$ by $\mu:=p+T^t q\in \mathbb{R}^n\oplus T^t\mathbb{R}^n$. In general, the affine Lagrangian submanifold corresponding to a holomorphic vector bundle $E_{(r,A,\mu ,\mathcal{U})}$ is the following (we will explain the details of the symplectic geometry side again later) :
\begin{equation*}
\left\{ \left( \begin{array}{ccc} \check{x} \\ \check{y} \end{array} \right) \in \check{T}^{2n}_{J=T} \ |\ \check{y}=\frac{1}{r}A\check{x}+\frac{1}{r}p \right\}.
\end{equation*}
Here, $\check{x}:=(x^1,\cdots,x^n)^t$, $\check{y}:=(y^1,\cdots,y^n)^t$ are the coordinates of the mirror partner $\check{T}^{2n}_{J=T}$ of the complex torus $T^{2n}_{J=T}$. In this situation, if $x^j\mapsto x^j+2\pi $ ($j=1,\cdots,n$), then 
\begin{equation*}
\check{y}\mapsto \check{y}+\frac{2\pi }{r}(a_{1j},\cdots, a_{nj})^t.
\end{equation*}
We decide the transition functions of $E_{(r,A,\mu ,\mathcal{U})}$ by using this $\frac{1}{r}(a_{1j},\cdots,a_{nj})^t \in \mathbb{Q}^n$. This construction is a generalization of the case of elliptic curves $(T^2_{J=T},\check{T}^2_{J=T})$ to the higher dimensional case in the paper \cite{kazushi} (see section 2). Now, we give the strict definition of $E_{(r,A,\mu,\mathcal{U})}$. We define $r'\in \mathbb{N}$ by using a given pair $(r,A)\in \mathbb{N}\times M(n;\mathbb{Z})$ as follows. By the theory of elementary divisors, there exist two matrices $\mathcal{A}$, $\mathcal{B}\in GL(n;\mathbb{Z})$ such that 
\begin{equation}
\mathcal{A}A\mathcal{B}=\left( \begin{array}{cccccc} \tilde{a_1} & & & & & \\ & \ddots & & & & \\ & & \tilde{a_s} & & & \\ & & & 0 & & \\ & & & & \ddots & \\ & & & & & 0 \end{array} \right), \label{matAB}
\end{equation}
where $\tilde{a_i}\in \mathbb{N}$ ($i=1,\cdots, s, 1\leq s\leq n$) and $\tilde{a_i}|\tilde{a_{i+1}}$ ($i=1,\cdots, s-1$). Then, we define $r_i'\in \mathbb{N}$ and $a_i'\in \mathbb{Z}$ ($i=1,\cdots, s$) by
\begin{equation*}
\frac{\tilde{a_i}}{r}=\frac{a_i'}{r_i'}, \ gcd(r_i',a_i')=1,
\end{equation*}
where $gcd(m,n)>0$ denotes the greatest common divisor of $m$, $n\in \mathbb{Z}$. By using these, we set
\begin{equation}
r':=r_1'\cdots r_s'\in \mathbb{N}. \label{r'}
\end{equation}
This $r'\in \mathbb{N}$ is uniquely defined by a given pair $(r,A)\in \mathbb{N}\times M(n;\mathbb{Z})$, and it is actually the rank of $E_{(r,A,\mu,\mathcal{U})}$ (in this sense, although we should also emphasize $r'\in \mathbb{N}$ when we denote $E_{(r,A,\mu,\mathcal{U})}$, for simplicity, we use the notation $E_{(r,A,\mu,\mathcal{U})}$ in this paper). Let
\begin{equation*}
\psi ^{l_1 \cdots l_n}_{m_1 \cdots m_n} : O^{l_1 \cdots l_n}_{m_1 \cdots m_n}\rightarrow O^{l_1 \cdots l_n}_{m_1 \cdots m_n} \times \mathbb{C}^{r'}, \hspace{5mm}l_j, m_k =1,2,3
\end{equation*}
be a smooth section of $E_{(r,A,\mu ,\mathcal{U})}|_{O^{l_1 \cdots l_n}_{m_1 \cdots m_n}}$. The transition functions of $E_{(r,A,\mu ,\mathcal{U})}$ are non-trivial on 
\begin{align*}
&O^{(l_1=3) \cdots l_n}_{m_1 \cdots m_n}\cap O^{(l_1=1) \cdots l_n}_{m_1 \cdots m_n}, \ O^{l_1 (l_2=3) \cdots l_n}_{m_1 \cdots m_n}\cap O^{l_1 (l_2=1) \cdots l_n}_{m_1 \cdots m_n},\cdots, O^{l_1 \cdots (l_n=3)}_{m_1 \cdots m_n}\cap O^{l_1 \cdots (l_n=1)}_{m_1 \cdots m_n},\\
&O^{l_1 \cdots l_n}_{(m_1=3) \cdots m_n}\cap O^{l_1 \cdots l_n}_{(m_1=1) \cdots m_n}, \ O^{l_1 \cdots l_n}_{m_1 (m_2=3) \cdots m_n}\cap O^{l_1 \cdots l_n}_{m_1 (m_2=1) \cdots m_n},\cdots, \\
&O^{l_1 \cdots l_n}_{m_1 \cdots (m_n=3)}\cap O^{l_1 \cdots l_n}_{m_1 \cdots (m_n=1)},
\end{align*}
and otherwise are trivial. We define the transition function on $O^{l_1 \cdots (l_j =3) \cdots l_n}_{m_1 \cdots m_n}\cap O^{l_1 \cdots (l_j =1) \cdots l_n}_{m_1 \cdots m_n}$ by
\begin{align*}
&\left.\psi ^{l_1 \cdots (l_j =3) \cdots l_n}_{m_1 \cdots m_n} \right|_{O^{l_1 \cdots (l_j =3) \cdots l_n}_{m_1 \cdots m_n}\cap O^{l_1 \cdots (l_j =1) \cdots l_n}_{m_1 \cdots m_n}}\\
&=e^{\frac{\mathbf{i}}{r}a_j y}V_j \left.\psi ^{l_1 \cdots (l_j =1) \cdots l_n}_{m_1 \cdots m_n} \right|_{O^{l_1 \cdots (l_j =3) \cdots l_n}_{m_1 \cdots m_n}\cap O^{l_1 \cdots (l_j =1) \cdots l_n}_{m_1 \cdots m_n}},
\end{align*}
where $\mathbf{i}=\sqrt{-1}$, $a_j:=(a_{1j},\cdots,a_{nj})\in \mathbb{Z}^n$ and $V_j\in U(r')$. Similarly, we define the transition function on $O^{l_1 \cdots l_n}_{m_1 \cdots (m_k =3) \cdots m_n}\cap O^{l_1 \cdots l_n}_{m_1 \cdots (m_k=1) \cdots m_n}$ by
\begin{align*}
&\left.\psi ^{l_1 \cdots l_n}_{m_1 \cdots (m_k=3) \cdots m_n} \right|_{O^{l_1 \cdots l_n}_{m_1 \cdots (m_k =3) \cdots m_n}\cap O^{l_1 \cdots l_n}_{m_1 \cdots (m_k=1) \cdots m_n}}\\
&=U_k \left.\psi ^{l_1 \cdots l_n}_{m_1 \cdots (m_k=1) \cdots m_n} \right|_{O^{l_1 \cdots l_n}_{m_1 \cdots (m_k =3) \cdots m_n}\cap O^{l_1 \cdots l_n}_{m_1 \cdots (m_k=1) \cdots m_n}},
\end{align*}
where $U_k\in U(r')$. In the definition of these transition functions, actually, we only treat $V_j$, $U_k\in U(r')$ which satisfy the cocycle condition, so we explain the cocycle condition below. When we define
\begin{align*}
&\left.\psi ^{l_1 \cdots (l_j =3) \cdots l_n}_{m_1 \cdots (m_k=3) \cdots m_n} \right|_{O^{l_1 \cdots (l_j =3) \cdots l_n}_{m_1 \cdots (m_k =3) \cdots m_n}\cap O^{l_1 \cdots (l_j =1) \cdots l_n}_{m_1 \cdots (m_k=1) \cdots m_n}}\\
&=U_k \left.\psi ^{l_1 \cdots (l_j =3) \cdots l_n}_{m_1 \cdots (m_k=1) \cdots m_n} \right|_{O^{l_1 \cdots (l_j =3) \cdots l_n}_{m_1 \cdots (m_k =3) \cdots m_n}\cap O^{l_1 \cdots (l_j =1) \cdots l_n}_{m_1 \cdots (m_k=1) \cdots m_n}}\\
&=\Bigl (U_k\Bigr )\left (e^{\frac{\mathbf{i}}{r}a_j y}V_j\right ) \left.\psi ^{l_1 \cdots (l_j =1) \cdots l_n}_{m_1 \cdots (m_k=1) \cdots m_n} \right|_{O^{l_1 \cdots (l_j =3) \cdots l_n}_{m_1 \cdots (m_k =3) \cdots m_n}\cap O^{l_1 \cdots (l_j =1) \cdots l_n}_{m_1 \cdots (m_k=1) \cdots m_n}},
\end{align*}
the cocycle condition is expressed as
\begin{equation*}
V_j V_k=V_k V_j,\ U_j U_k=U_k U_j,\ \zeta  ^{-a_{kj}}U_k V_j=V_j U_k,
\end{equation*}
where $\zeta :=e^{\frac{2\pi \mathbf{i}}{r}}$, and $j,k=1,\cdots,n$. We define a set $\mathcal{U}$ of unitary matrices by 
\begin{align}
&\mathcal{U}:= \Bigl \{ V_j , U_k \in U(r') \ | \ V_j V_k=V_k V_j,\ U_j U_k=U_k U_j,\ \zeta ^{-a_{kj}}U_k V_j=V_j U_k, \notag \\
& \hspace{42mm} j,k=1,\cdots ,n \Bigr \}. \label{setU}
\end{align}
Of course, how to define the set $\mathcal{U}$ relates closely to (in)decomposability of $E_{(r,A,\mu,\mathcal{U})}$. Here, we only treat the set $\mathcal{U}$ such that $E_{(r,A,\mu,\mathcal{U})}$ is simple. Actually, we can take such a set $\mathcal{U}\not=\emptyset $ for any $(r,A,r')\in \mathbb{N}\times M(n;\mathbb{Z})\times \mathbb{N}$, and this fact is proved in \cite[Proposition 3.2]{kazushi3}. Furthermore, we define a connection $\nabla_{(r,A,\mu ,\mathcal{U})}$ on $E_{(r,A,\mu ,\mathcal{U})}$ locally as 
\begin{align*}
\nabla_{(r,A,\mu ,\mathcal{U})}&=d+\omega _{(r,A,\mu ,\mathcal{U})} \\
&:=d-\frac{\mathbf{i}}{2\pi }\left(\frac{1}{r}x^t A^t +\frac{1}{r}\mu ^t\right)dy\cdot I_{r'} \\
&=d-\frac{\mathbf{i}}{2\pi }\left( \left( \frac{1}{r}x^t A^t +\frac{1}{r}p \right)+\frac{1}{r}q^t T \right)dy\cdot I_{r'}, 
\end{align*}
where $dy:=(dy_1,\cdots,dy_n)^t$ and $d$ denotes the exterior derivative. In fact, $\nabla_{(r,A,\mu ,\mathcal{U})}$ is compatible with the transition functions and so defines a global connection. Then, its curvature form $\Omega _{(r,A,\mu ,\mathcal{U})}$ is expressed locally as
\begin{equation*}
\Omega _{(r,A,\mu ,\mathcal{U})}=-\frac{\mathbf{i}}{2\pi r}dx^t A^t dy\cdot I_{r'},
\end{equation*}
where $dx:=(dx_1,\cdots,dx_n)^t$. Here, we consider the condition such that $E_{(r,A,\mu ,\mathcal{U})}$ is holomorphic. We see that the following proposition holds.
\begin{proposition}\label{pro2.1}
For a given quadruple $(r,A,p,q)\in \mathbb{N}\times M(n;\mathbb{Z})\times \mathbb{R}^n\times\mathbb{R}^n$, the complex vector bundle $E_{(r,A,\mu ,\mathcal{U})}\rightarrow T^{2n}_{J=T}$ is holomorphic if and only if $AT=(AT)^t$ holds.
\end{proposition}
\begin{proof}
A complex vector bundle is holomorphic if and only if the (0,2)-part of its curvature form vanishes, so we calculate the (0,2)-part of $\Omega _{(r,A,\mu ,\mathcal{U})}$. It turns out to be
\begin{equation*}
\Omega ^{(0,2)}_{(r,A,\mu ,\mathcal{U})}=\frac{\mathbf{i}}{2\pi r}d\bar{z}^t\{ T(T-\bar{T} )^{-1} \}^t A^t (T-\bar{T})^{-1}d\bar{z}\cdot I_{r'},
\end{equation*}
where $d\bar{z} :=(d\bar{z}_1,\cdots,d\bar{z}_n)^t$. Thus, $\Omega ^{(0,2)}_{(r,A,\mu ,\mathcal{U})}=0$ is equivalent to that $\{ T(T-\bar{T} )^{-1} \}^t A^t (T-\bar{T} )^{-1}$ is a symmetric matrix, i.e., $AT=(AT)^t$.
\end{proof}

Next, we explain the symplectic geometry side. Let us consider the $2n$-dimensional standard real torus $T^{2n}=\mathbb{R}^{2n}/2\pi \mathbb{Z}^{2n}$. For each point $(x^1,\cdots, x^n, \\ y^1,\cdots, y^n)^t \in T^{2n}$, we identify $x^i\sim x^i+2\pi $, $y^i\sim y^i+2\pi $, where $i=1,\cdots, n$. We also denote by $(x^1,\cdots, x^n, y^1,\cdots, y^n)^t$ the local coordinates in the neighborhood of an arbitrary point $(x^1,\cdots, x^n, y^1,\cdots, y^n)^t\in T^{2n}$. Furthermore, we use the same notation $(x^1,\cdots, x^n, y^1,\cdots, y^n)^t$ when we denote the coordinates of the covering space $\mathbb{R}^{2n}$ of $T^{2n}$. For simplicity, we set
\begin{equation*}
\check{x}:=(x^1,\cdots, x^n)^t, \ \check{y}:=(y^1,\cdots, y^n)^t.
\end{equation*}
We define a complexified symplectic form $\tilde{\omega } $ on $T^{2n}$ by
\begin{equation*}
\tilde{\omega }  :=d\check{x}^t (-T^{-1})^t d\check{y},
\end{equation*}
where $d\check{x}:=(dx^1,\cdots,dx^n)^t$ and $d\check{y}:=(dy^1,\cdots,dy^n)^t$. We decompose $\tilde{\omega }$ into 
\begin{equation*}
\tilde{\omega } =d\check{x}^t \mathrm{Re}(-T^{-1})^t d\check{y} +\mathbf{i} d\check{x}^t \mathrm{Im}(-T^{-1})^t d\check{y},
\end{equation*}
and define 
\begin{equation*}
\omega :=\mathrm{Im}(-T^{-1})^t , \ B:=\mathrm{Re}(-T^{-1})^t.
\end{equation*}
Sometimes we identify the matrices $\omega $ and $B$ with the 2-forms $d\check{x}^t \omega d\check{y}$ and $d\check{x}^t B d\check{y}$, respectively. Then, $\omega $ gives a symplectic form on $T^{2n}$. The closed 2-form $B$ is often called the $B$-field. This complexified symplectic torus $(T^{2n}, \tilde{\omega}=d\check{x}^t (-T^{-1})^t d\check{y})$ is a mirror partner of the complex torus $T^{2n}_{J=T}$. Hereafter, we denote
\begin{equation*}
\check{T}^{2n}_{J=T}:=(T^{2n}, \tilde{\omega}=d\check{x}^t (-T^{-1})^t d\check{y})
\end{equation*}
for simplicity. We define the objects of the Fukaya category on $\check{T}^{2n}_{J=T}$ corresponding to holomorphic vector bundles $E_{(r,A,\mu ,\mathcal{U})}\rightarrow T^{2n}_{J=T}$, namely, the pairs of affine Lagrangian submanifolds in $\check{T}^{2n}_{J=T}$ and unitary local systems along them. First, we recall the definition of objects of the Fukaya categories following \cite[Definition 1.1]{Fuk}. Let $(M, \Omega)$ be a symplectic manifold $(M, \omega)$ together with a closed 2-form $B$ on $M$. Here, we put $\Omega=\omega+\sqrt{-1}B$ (note $-B+\sqrt{-1}\omega$ is used in many of the literatures). Then, we consider pairs $(L, \mathcal{L})$ with the following properties :
\begin{align}
&L \ \mathrm{is} \ \mathrm{a} \ \mathrm{Lagrangian} \ \mathrm{submanifold} \ \mathrm{of} \ (M,\omega). \label{f1} \\
&\mathcal{L}\rightarrow L \ \mathrm{is} \ \mathrm{a} \ \mathrm{line} \ \mathrm{bundle} \ \mathrm{together} \ \mathrm{with} \ \mathrm{a} \ \mathrm{connection} \ \nabla^{\mathcal{L}} \ \mathrm{such} \ \mathrm{that} \label{f2} \\
&F_{\nabla^{\mathcal{L}}}=2\pi\sqrt{-1}B|_L. \notag
\end{align}
In this context, $F_{\nabla^{\mathcal{L}}}$ denotes the curvature form of the connection $\nabla^{\mathcal{L}}$. We define objects of the Fukaya category on $(M, \Omega)$ by pairs $(L,\mathcal{L})$ which satisfy the properties (\ref{f1}), (\ref{f2}). Let us consider the following $n$-dimensional submanifold $\tilde{L}_{(r,A,p)}$ in $\mathbb{R}^{2n}$ :
\begin{equation*}
\tilde{L} _{(r,A,p)}:=\left\{ \left( \begin{array}{ccc} \check{x} \\ \check{y} \end{array} \right)\in \mathbb{R}^{2n} \ | \ \check{y}=\frac{1}{r}A\check{x}+\frac{1}{r}p \right\}.
\end{equation*}
We see that this $n$-dimensional submanifold $\tilde{L} _{(r,A,p)}$ satisfies the property (\ref{f1}), namely, $\tilde{L}_{(r,A,p)}$ becomes a Lagrangian submanifold in $\mathbb{R}^{2n}$ if and only if $\omega A=(\omega A)^t$ holds. Then, for the covering map $\pi : \mathbb{R}^{2n}\rightarrow \check{T}^{2n}_{J=T}$, 
\begin{equation*}
L_{(r,A,p)}:=\pi (\tilde{L}_{(r,A,p)})
\end{equation*}
defines a Lagrangian submanifold in $\check{T}^{2n}_{J=T}$. On the other hand, we can also regard the complexified symplectic torus $\check{T}^{2n}_{J=T}$ as the trivial special Lagrangian torus fibration $\check{\pi } : \check{T}^{2n}_{J=T} \rightarrow \mathbb{R}^n/2\pi \mathbb{Z}^n$, where $\check{x}$ is the local coordinates of the base space $\mathbb{R}^n/2\pi \mathbb{Z}^n$ and $\check{y}$ is the local coordinates of the fiber of $\check{\pi } : \check{T}^{2n}_{J=T} \rightarrow \mathbb{R}^n/2\pi \mathbb{Z}^n$. Then, we can interpret each affine Lagrangian submanifold $L_{(r,A,p)}$ in $\check{T}^{2n}_{J=T}$ as the affine Lagrangian multi section 
\begin{equation*}
s(\check{x})=\frac{1}{r}A\check{x}+\frac{1}{r}p
\end{equation*}
of $\check{\pi } : \check{T}^{2n}_{J=T} \rightarrow \mathbb{R}^n/2\pi \mathbb{Z}^n$.
\begin{rem}
As explained above, while $r':=r_1'\cdots r_s'\in \mathbb{N}$ is the rank of $E_{(r,A,\mu,\mathcal{U})}\rightarrow T^{2n}_{J=T}$ $($see the relations $(\ref{matAB})$ and $(\ref{r'}))$, in the symplectic geometry side, this $r'\in \mathbb{N}$ is interpreted as follows. For the affine Lagrangian submanifold $L_{(r,A,p)}$ in $\check{T}^{2n}_{J=T}$ which is defined by a given data $(r,A,p)\in \mathbb{N}\times M(n;\mathbb{Z})\times \mathbb{R}^n$, we regard it as the affine Lagrangian multi section $s(\check{x})=\frac{1}{r}A\check{x}+\frac{1}{r}p$ of $\check{\pi } : \check{T}^{2n}_{J=T} \rightarrow \mathbb{R}^n/2\pi \mathbb{Z}^n$. Then, for each point $\check{x}\in \mathbb{R}^n/2\pi \mathbb{Z}^n$, we see
\begin{align*}
s(\check{x})=\biggl\{ & \left( \frac{1}{r}A\check{x}+\frac{1}{r}p+\frac{2\pi }{r}A\mathcal{B}M_s \right) \in \check{\pi }^{-1}(\check{x})\approx \mathbb{R}^n/2\pi \mathbb{Z}^n \ | \\
&M_s=(m_1,\cdots, m_s, 0,\cdots, 0)^t\in \mathbb{Z}^n, \ 0\leq m_i\leq r_i'-1, \ i=1,\cdots, s \biggr\},
\end{align*}
and this indicates that $s(\check{x})$ consists of $r'$ points. Thus, we can regard $r'\in \mathbb{N}$ as the multiplicity of $s(\check{x})=\frac{1}{r}A\check{x}+\frac{1}{r}p$.
\end{rem}
We then consider the trivial complex line bundle 
\begin{equation*}
\mathcal{L}_{(r,A,p,q)}\rightarrow L_{(r,A,p)}
\end{equation*}
with the flat connection
\begin{equation*}
\nabla_{\mathcal{L}_{(r,A,p,q)}}:=d-\frac{\mathbf{i}}{2\pi }\frac{1}{r}q^t d\check{x},
\end{equation*}
where $q\in \mathbb{R}^n$ is the unitary holonomy of $\mathcal{L}_{(r,A,p,q)}$ along $L_{(r,A,p)}\approx T^n$. We discuss the property (\ref{f2}) for this pair $(L_{(r,A,p)}, \mathcal{L}_{(r,A,p,q)})$ :
\begin{equation*}
\Omega _{\mathcal{L}_{(r,A,p,q)}}=\left. d\check{x}^t B d\check{y} \right|_{L_{(r,A,p)}}.
\end{equation*}
Here, $\Omega _{\mathcal{L}_{(r,A,p,q)}}$ is the curvature form of the flat connection $\nabla_{\mathcal{L}_{(r,A,p,q)}}$, i.e., $\Omega _{\mathcal{L}_{(r,A,p,q)}}=0$. Hence, we see 
\begin{equation*}
\left. d\check{x}^t B d\check{y}\right|_{L_{(r,A,p)}}=\frac{1}{r}d\check{x}^t BA d\check{x}=0,
\end{equation*}
so one has $BA=(BA)^t$. Note that $\omega A=(\omega A)^t$ and $BA=(BA)^t$ hold if and only if $AT=(AT)^t$ holds. By summarizing the above discussions, we obtain the following proposition. In particular, the condition $AT=(AT)^t$ in the following proposition is also the condition such that a complex vector bundle $E_{(r,A,\mu,\mathcal{U})}\rightarrow T^{2n}_{J=T}$ becomes a holomorphic vector bundle (see Proposition \ref{pro2.1}).
\begin{proposition} \label{propfukob}
For a given quadruple $(r,A,p,q)\in \mathbb{N}\times M(n;\mathbb{Z})\times \mathbb{R}^n\times \mathbb{R}^n$, $(L_{(r,A,p)}, \mathcal{L}_{(r,A,p,q)})$ gives an object of the Fukaya category on $\check{T}^{2n}_{J=T}$ if and only if $AT=(AT)^t$ holds.
\end{proposition}
\begin{definition}
We denote the full subcategory of the Fukaya category on $\check{T}^{2n}_{J=T}$ consisting of objects $(L_{(r,A,p)}, \mathcal{L}_{(r,A,p,q)})$ which satisfy the condition $AT=(AT)^t$ by $Fuk_{\rm aff}(\check{T}^{2n}_{J=T})$.
\end{definition}

We define a DG-category 
\begin{equation*}
DG_{T^{2n}_{J=T}}
\end{equation*}
consisting of holomorphic vector bundles $(E_{(r,A,\mu ,\mathcal{U})},\nabla_{(r,A,\mu ,\mathcal{U})})$. This definition is an extension of the case of elliptic curves $(T^2_{J=T},\check{T}^2_{J=T})$ to the higher dimensional case in the paper \cite{kazushi} (see section 3). The objects of $DG_{T^{2n}_{J=T}}$ are holomorphic vector bundles $E_{(r,A,\mu ,\mathcal{U})}$ with $U(r')$-connections $\nabla_{(r,A,\mu ,\mathcal{U})}$. Of course, we assume $AT=(AT)^t$. Sometimes we simply denote $(E_{(r,A,\mu ,\mathcal{U})},\nabla_{(r,A,\mu ,\mathcal{U})})$ by $E_{(r,A,\mu ,\mathcal{U})}$. For any two objects 
\begin{equation*}
E_{(r,A,\mu ,\mathcal{U})}=(E_{(r,A,\mu ,\mathcal{U})},\nabla_{(r,A,\mu ,\mathcal{U})}), \ E_{(s,B,\nu ,\mathcal{V})}=(E_{(s,B,\nu ,\mathcal{V})},\nabla_{(s,B,\nu ,\mathcal{V})}), 
\end{equation*}
the space of morphisms is defined by
\begin{equation*}
\mathrm{Hom}_{DG_{T^{2n}_{J=T}}}(E_{(r,A,\mu ,\mathcal{U})},E_{(s,B,\nu ,\mathcal{V})}):=\Gamma (E_{(r,A,\mu ,\mathcal{U})},E_{(s,B,\nu ,\mathcal{V})})\bigotimes _{C^{\infty }(T^{2n}_{J=T})}\Omega ^{0,*}(T^{2n}_{J=T}),
\end{equation*}
where $\Omega ^{0,*}(T^{2n}_{J=T})$ is the space of anti-holomorphic differential forms, and 
\begin{equation*}
\Gamma (E_{(r,A,\mu ,\mathcal{U})},E_{(s,B,\nu ,\mathcal{V})})
\end{equation*}
is the space of homomorphisms from $E_{(r,A,\mu ,\mathcal{U})}$ to $E_{(s,B,\nu ,\mathcal{V})}$. The space of morphisms $\mathrm{Hom}_{DG_{T^{2n}_{J=T}}}(E_{(r,A,\mu ,\mathcal{U})},E_{(s,B,\nu ,\mathcal{V})})$ is a $\mathbb{Z}$-graded vector space, where the grading is defined as the degree of the anti-holomorphic differential forms. The degree $r$ part is denoted $\mathrm{Hom}^r_{DG_{T^{2n}_{J=T}}}(E_{(r,A,\mu ,\mathcal{U})},E_{(s,B,\nu ,\mathcal{V})})$. We decompose $\nabla_{(r,A,\mu ,\mathcal{U})}$ into its holomorphic part and anti-holomorphic part $\nabla_{(r,A,\mu ,\mathcal{U})}=\nabla ^{(1,0)}_{(r,A,\mu ,\mathcal{U})}+\nabla ^{(0,1)}_{(r,A,\mu ,\mathcal{U})}$, and define a linear map
\begin{equation*}
\mathrm{Hom}^r_{DG_{T^{2n}_{J=T}}}(E_{(r,A,\mu ,\mathcal{U})},E_{(s,B,\nu ,\mathcal{V})})\rightarrow \mathrm{Hom}^{r+1}_{DG_{T^{2n}_{J=T}}}(E_{(r,A,\mu ,\mathcal{U})},E_{(s,B,\nu ,\mathcal{V})})
\end{equation*}
by
\begin{equation*}
\psi \mapsto ( 2\nabla^{(0,1)}_{(s,B,\nu ,\mathcal{V})} )(\psi )-(-1)^r\psi  (2\nabla ^{(0,1)}_{(r,A,\mu ,\mathcal{U})} ).
\end{equation*}
We can check that this linear map is a differential. Furthermore, the product structure is defined by the composition of homomorphisms of vector bundles together with the wedge product for the anti-holomorphic differential forms. Then, these differential and product structure satisfy the Leibniz rule. Thus, $DG_{T^{2n}_{J=T}}$ forms a DG-category.
\begin{rem}
In general, for any $A_{\infty}$-category $\mathscr{C}$, we can construct a triangulated category $Tr(\mathscr{C})$ by using the Bondal-Kapranov-Kontsevich construction \cite{bondal}, \cite{Kon}. We expect that the DG-category $DG_{T^{2n}_{J=T}}$ generates the bounded derived category of coherent sheaves $D^b(Coh(T^{2n}_{J=T}))$ on $T^{2n}_{J=T}$ in the sense of the Bondal-Kapranov-Kontsevich construction, i.e.,
\begin{equation*}
Tr(DG_{T^{2n}_{J=T}})\cong D^b(Coh(T^{2n}_{J=T})).
\end{equation*}
At least, it is known that it split generates $D^b(Coh(T^{2n}_{J=T}))$ when $T^{2n}_{J=T}$ is an abelian variety (cf. \cite{orlov}, \cite{abouzaid}).
\end{rem}
On the correspondence between two $A_{\infty}$(DG)-categories $DG_{T^{2n}_{J=T}}$ and $Fuk_{\rm aff}(\check{T}^{2n}_{J=T})$, it is known that the following theorem holds (\cite[Theorem 5.1]{kazushi3}). Note that two parameters $\theta$, $\xi\in \mathbb{R}^n$ in the following theorem are defined as follows. For elements $V_j$, $U_k\in \mathcal{U}$ $(j,k=1,\cdots, n)$, Let us define $\xi_j$, $\theta_k\in \mathbb{R}$ by
\begin{equation*}
e^{\mathbf{i}\xi_j}=\mathrm{det}V_j, \ e^{\mathbf{i}\theta_k}=\mathrm{det}U_k.
\end{equation*}
Then, we set
\begin{equation*}
\xi:=(\xi_1,\cdots, \xi_n)^t, \ \theta:=(\theta_1,\cdots, \theta_n)^t.
\end{equation*}
\begin{theo}\label{bijection1}
A map $\mathrm{Ob}(DG_{T^{2n}_{J=T}}) \rightarrow \mathrm{Ob}(Fuk_{\rm aff}(\check{T}^{2n}_{J=T}))$ is defined by
\begin{equation*}
E_{(r,A,\mu, \mathcal{U})} \mapsto (L_{(r,A,p-\frac{r}{r'}\theta)}, \mathcal{L}_{(r,A,p-\frac{r}{r'}\theta,q+\frac{r}{r'}\xi)}),
\end{equation*}
and it induces a bijection between $\mathrm{Ob}^{isom}(DG_{T^{2n}_{J=T}})$ and $\mathrm{Ob}^{isom}(Fuk_{\rm aff}(\check{T}^{2n}_{J=T}))$, where $\mathrm{Ob}^{isom}(DG_{T^{2n}_{J=T}})$ and $\mathrm{Ob}^{isom}(Fuk_{\rm aff}(\check{T}^{2n}_{J=T}))$ denote the set of the isomorphism classes of objects of $DG_{T^{2n}_{J=T}}$ and the set of the isomorphism classes of objects of $Fuk_{\rm aff}(\check{T}^{2n}_{J=T})$, respectively\footnote{We consider affine Lagrangian submanifolds only in this paper, so two objects $(L_{(r,A,p)},\mathcal{L}_{(r,A,p,q)})$, $(L_{(s,B,u)},\mathcal{L}_{(s,B,u,v)})\in Fuk_{\rm aff}(\check{T}^{2n}_{J=T})$ are isomorphic to each other if and only if $L_{(r,A,p)}=L_{(s,B,u)}$ and $\mathcal{L}_{(r,A,p,q)}\cong \mathcal{L}_{(s,B,u,v)}$.}.
\end{theo}

\section{The construction of an isomorphism $E_{(r,A,\mu ,\mathcal{U})}\cong \mathcal{E}_{(r,A,\mu ,\mathcal{U} )}$}
In this section, we first recall the definition of projectively flat bundles and some properties of them. Next, we construct a one-to-one correspondence between holomorphic vector bundles $E_{(r,A,\mu ,\mathcal{U})}$ and a certain kind of projectively flat bundles. In general, factors of automorphy of projectively flat bundles on complex tori are classified concretely, so we interpret holomorphic vector bundles $E_{(r,A,\mu ,\mathcal{U})}$ in the language of those factors of automorphy. This result is given in Theorem \ref{theo3.6}.

We recall the definition of factors of automorphy for holomorphic vector bundles following \cite{koba}. Let $M$ be a complex manifold such that its universal covering space $\tilde{M} $ is a topologically trivial (contractible) Stein manifold ($\mathbb{C}^n$ is an example of a Stein manifold). Let $p : \tilde{M} \rightarrow M$ be the covering projection and $\Gamma $ the covering transformation group acting on $\tilde{M} $ so that $M=\tilde{M} /\Gamma $. Let $E$ be a holomorphic vector bundle of rank $r$ over $M$. Then its pull-back $\tilde{E}=p^*E$ is a holomorphic vector bundle of the same rank over $\tilde{M} $. Since $\tilde{M} $ is topologically trivial, $\tilde{E} $ is topologically a product bundle. Since $\tilde{M} $ is Stein, by Oka's principle, $\tilde{E} $ is holomorphically a product bundle, i.e., $\tilde{E}=\tilde{M} \times \mathbb{C}^r$. Having fixed this isomorphism, we define a holomorphic map $j : \Gamma \times \tilde{M} \rightarrow GL(r;\mathbb{C})$ by the commutative diagram
\begin{equation*}
\xymatrix{ \tilde{E}_{\gamma (x)}\cong \mathbb{C}^r \ar[dr]& &\tilde{E}_x \cong \mathbb{C}^r \ar[dl] \ar[ll]_{j(\gamma ,x)}\\
& E_{p(x)} & ,}
\end{equation*}
where $x\in \tilde{M} $, $\gamma \in \Gamma $. Then, for $x\in \tilde{M} $, $\gamma ,\gamma '\in \Gamma $, the relation 
\begin{equation*}
j(\gamma +\gamma ',x)=j(\gamma ',x+\gamma )\circ j(\gamma ,x)
\end{equation*}
holds. The map $j : \Gamma \times \tilde{M} \rightarrow GL(r;\mathbb{C})$ is called the factor of automorphy for the holomorphic vector bundle $E$.

Now, we recall the definition and some properties of projectively flat bundles.
\begin{definition}[Projectively flat bundles, \cite{Hano}, \cite{matsu}, \cite{koba}, \cite{yang}]
Let $E$ be a holomorphic vector bundle of rank $r$ over a compact K\"{a}hler manifold $M$ and $P(E)$ its associated principal $GL(r;\mathbb{C})$-bundle. Then $\hat{P}(E)=P(E)/\mathbb{C}^{\times} I_r$ is a principal $PGL(r;\mathbb{C})$-bundle. We say that $E$ is projectively flat when $\hat{P}(E)$ is provided with a flat structure.
\end{definition}
For a complex vector bundle $E$ of rank $r$ with a connection $D$ over a compact K\"{a}hler manifold $M$, it is known that the following proposition holds.
\begin{proposition}[\cite{matsu}, \cite{koba}, \cite{yang}]\label{pro3.2}
Let $R$ be a curvature of $(E,D)$. Then, $E$ is projectively flat if and only if $R$ takes values in scalar multiples of the identity endmorphism $I_E$ of $E$, i.e., if and only if there exists a complex 2-form $\alpha $ on $M$ such that $R=\alpha \cdot I_E$.
\end{proposition}
There are many studies of projectively flat bundles on complex tori, i.e., the cases $M=\mathbb{C}^n/\Gamma $, where $\Gamma $ is a nondegenerate lattice of rank $2n$ in $\mathbb{C}^n$ (\cite{Hano}, \cite{matsu}, \cite{koba}, \cite{yang} etc.). Let us denote the coordinates of $\mathbb{C}^n$ by $z=(z_1,\cdots,z_n)^t$. Hereafter, we focus on projectively flat bundles which admit Hermitian structures\footnote{In fact, since we do not mention Hermitian structures explicitly in our main discussions, readers do not have to consider them so much in section 3.} over a complex torus $\mathbb{C}^n/\Gamma$. On the detail of the results which are described below, for example, see \cite{Hano}, \cite{matsu}, \cite{koba}, \cite{yang}. Now, we recall the following theorem (see \cite[Theorem 4.7.54]{koba}) which plays an important role in our main discussions in this section.
\begin{theo}\label{projflat}
Let $E$ be a holomorphic vector bundle of rank $r$ over a complex torus $\mathbb{C}^n/\Gamma$. If $E$ admits a projectively flat Hermitian structure $h$, then its factor of automorphy $j$ can be written as follows $:$
\begin{equation*}
j(\gamma ,z)=U(\gamma )\mathrm{exp}\left\{ \frac{1}{r}\mathcal{R}(z,\gamma )+\frac{1}{2r}\mathcal{R}(\gamma ,\gamma ) \right\} \ \ (\gamma, z)\in \Gamma \times \mathbb{C}^n,
\end{equation*}
where \\\\
$(\mathrm{i})$ $\mathcal{R}$ is a Hermitian form on $\mathbb{C}^n$ and its imaginary part satisfies
\begin{equation*}
\mathrm{Im}\mathcal{R}(\gamma, \gamma')\in \pi \mathbb{Z} \  \ for \  \ \gamma, \gamma'\in \Gamma,
\end{equation*}
$(\mathrm{ii})$ $U : \Gamma \rightarrow U(r)$ is a semi-representation in the sense that it satisfies
\begin{equation*}
U(\gamma +\gamma ')=U(\gamma )U(\gamma ')e^{\frac{\mathbf{i}}{r}\mathrm{Im}\mathcal{R}(\gamma ' ,\gamma )} \ \ for \ \ \gamma, \gamma'\in \Gamma.
\end{equation*}
Conversely, given a Hermitian form $\mathcal{R}$ on $\mathbb{C}^n$ with property $(\mathrm{i})$ and a semi-representation $U : \Gamma \rightarrow U(r)$, we can define a factor of automorphy $j : \Gamma \times \mathbb{C}^n \rightarrow CU(r)$ as above, where
\begin{equation*}
CU(r):=\left\{ cU \ | \ c\in \mathbb{C}^{\times} \ and \ U\in U(r) \right\}.
\end{equation*}
The corresponding vector bundle $E$ over $\mathbb{C}^n/\Gamma$ admits a projectively flat Hermitian structure.
\end{theo}
In Theorem \ref{projflat}, of course, by using a Hermitian matrix $R$, we can denote
\begin{equation*}
\mathcal{R}(z,w)=z^t R \bar{w},
\end{equation*}
where $z=(z_1,\cdots, z_n)^t$ and $w=(w_1,\cdots, w_n)^t$. Then, under the situation of Theorem \ref{projflat}, the connection 1-form $\omega$ of the Hermitian connection of $(E, h)$ is expressed locally as 
\begin{equation*}
\omega =-\frac{1}{r}\mathcal{R}(dz,z)\cdot I_r+dz^t b\cdot I_r,
\end{equation*}
where $dz:=(dz_1,\cdots,dz_n)^t$ and $b:=(b_1,\cdots,b_n)^t\in \mathbb{C}^n$ is a constant vector. Furthermore, the curvature form $\Omega$ of the Hermitian connection of $(E, h)$ is expressed locally as
\begin{equation*}
\Omega=\frac{1}{r}dz^t R d\bar{z}\cdot I_r.
\end{equation*}

We consider the case $T^{2n}_{J=T}=\mathbb{C}^n/2\pi (\mathbb{Z}^n \oplus T\mathbb{Z}^n)=\mathbb{C}^n/L$, and discuss the relations between holomorphic vector bundles $E_{(r,A,\mu ,\mathcal{U})}$ and projectively flat bundles. Note that the curvature form $\Omega _{(r,A,\mu ,\mathcal{U})}$ of a holomorphic vector bundle $E_{(r,A,\mu ,\mathcal{U})}$ is expressed locally as 
\begin{equation*}
\Omega _{(r,A,\mu ,\mathcal{U})}=\frac{\mathbf{i}}{2\pi r'}\frac{r'}{r}dz^t\{ (T-\bar{T} )^{-1} \}^t A d\bar{z} \cdot I_{r'}.
\end{equation*}
Now, we define
\begin{equation*}
R:=\frac{\mathbf{i}}{2\pi }\frac{r'}{r}\{ (T-\bar{T} )^{-1} \}^t A,
\end{equation*}
namely,
\begin{equation*}
\Omega _{(r,A,\mu ,\mathcal{U})}=\frac{1}{r'}dz^t R d\bar{z} \cdot I_{r'}.
\end{equation*}
Then, the following lemma holds.

\begin{lemma}
The matrix $R$ is a real symmetric matrix of order $n$.
\end{lemma}
\begin{proof}
By a direct calculation,
\begin{align*}
R&=\frac{\mathbf{i}}{2\pi }\frac{r'}{r}\{ (T-\bar{T} )^{-1} \}^t A (T-\bar{T} ) (T-\bar{T} )^{-1}\\
  &=\frac{\mathbf{i}}{2\pi }\frac{r'}{r}\{ (T-\bar{T} )^{-1} \}^t AT(T-\bar{T} )^{-1} -\frac{\mathbf{i}}{2\pi }\frac{r'}{r}\{ (T-\bar{T} )^{-1} \}^t A\bar{T} (T-\bar{T} )^{-1},
\end{align*}
and since $AT=(AT)^t$ holds, it is clear that the two matrices
\begin{equation*}
\frac{\mathbf{i}}{2\pi }\frac{r'}{r}\{ (T-\bar{T} )^{-1} \}^t AT(T-\bar{T} )^{-1},\ \frac{\mathbf{i}}{2\pi }\frac{r'}{r}\{ (T-\bar{T} )^{-1} \}^t A\bar{T} (T-\bar{T} )^{-1}
\end{equation*}
are symmetric. Hence, $R$ is a symmetric matrix. Furthermore, when we decompose $T=X+\mathbf{i}Y$ with $X:=\mathrm{Re}T$, $Y:=\mathrm{Im}T$, one has
\begin{equation*}
R=\frac{1}{4\pi }\frac{r'}{r}(Y^{-1})^t A.
\end{equation*}
This relation indicates $R\in M(n;\mathbb{R})$.
\end{proof}
\begin{rem}
Although the matrix $R$ is defined by using the matrix $\frac{r'}{r}A$, each component of the matrix $\frac{r'}{r}A$ is an integer. Actually, this matrix $\frac{r'}{r}A\in M(n;\mathbb{Z})$ corresponds to the 1-st Chern class of $E_{(r,A,\mu,\mathcal{U})}$.
\end{rem}
By using this real symmetric matrix $R=(R_{ij})$ of order $n$, we define a Hermitian bilinear form $\mathcal{R} : \mathbb{C}^n\times \mathbb{C}^n\rightarrow \mathbb{C}$ by
\begin{equation*}
\mathcal{R}(z,w):=\sum_{i,j=1}^n R_{ij} z_i \bar{w}_j,
\end{equation*}
where $z=(z_1,\cdots,z_n)^t$, $w=(w_1,\cdots,w_n)^t$. Then, the following propositions hold.
\begin{proposition}\label{pro3.4}
For $\gamma _1,\cdots,\gamma _n$ and $\gamma '_1,\cdots,\gamma '_n$, $\mathrm{Im}\mathcal{R}(\gamma _j,\gamma _k)=0$, $\mathrm{Im}\mathcal{R}(\gamma '_j,\gamma '_k)=0$, where $j,k=1,\cdots,n$.
\end{proposition}
\begin{proof}
By definition, $\mathcal{R}(\gamma _j,\gamma _k)=4\pi ^2R_{jk}$, where $R_{jk}\in \mathbb{R}$, so $\mathrm{Im}\mathcal{R}(\gamma _j,\gamma _k)=0$. On the other hand, we see $\mathcal{R}(\gamma '_j,\gamma '_k)=4\pi ^2(T^t R \bar{T} )_{jk}$, so for $T=X+\mathbf{i}Y$, it turns out to be
\begin{align*}
4\pi ^2 T^t R \bar{T} &=4\pi ^2 (X^t+\mathbf{i}Y^t)\cdot \frac{1}{4\pi }\frac{r'}{r}(Y^{-1})^t A\cdot (X-\mathbf{i}Y)\\
                              &=\pi \frac{r'}{r}\{ X^t (Y^{-1})^t AX+AY+\mathbf{i}(AX-X^t (Y^{-1})^t AY) \}.
\end{align*}
Thus, 
\begin{align*}
\mathrm{Im}\mathcal{R}(\gamma '_j,\gamma '_k)&=\left( \pi \frac{r'}{r}(AX-X^t (Y^{-1})^t AY) \right)_{jk}\\
                                                                  &=\left( \pi \frac{r'}{r}(AX-AX) \right)_{jk}\\
                                                                  &=O_{jk}.
\end{align*}
Here, the second equality follows from $AT=(AT)^t$.
\end{proof}
\begin{proposition}\label{pro3.5}
For $\gamma _1,\cdots,\gamma _n$ and $\gamma '_1,\cdots,\gamma '_n$, $\mathrm{Im}\mathcal{R}(\gamma _j,\gamma' _k)=-\pi \frac{r'}{r}a_{kj}$, $\mathrm{Im}\mathcal{R}(\gamma '_k,\gamma _j)=\pi \frac{r'}{r}a_{kj}$, where $j,k=1,\cdots,n$.
\end{proposition}
\begin{proof}
First, we prove $\mathrm{Im}R\bar{T} =-\frac{1}{4\pi }\frac{r'}{r}A^t$. For $T=X+\mathbf{i}Y$, 
\begin{equation*}
R\bar{T} =\frac{1}{4\pi } \frac{r'}{r}(Y^{-1})^t AX-\frac{\mathbf{i}}{4\pi }\frac{r'}{r}(Y^{-1})^t AY,
\end{equation*}
so we see
\begin{equation*}
\mathrm{Im}R\bar{T} =-\frac{1}{4\pi }\frac{r'}{r}(Y^{-1})^t AY=-\frac{1}{4\pi }\frac{r'}{r}A^t.
\end{equation*}
Here, we used $AT=(AT)^t$. Similarly, we can also prove 
\begin{equation*}
\mathrm{Im}\bar{R} T=\frac{1}{4\pi }\frac{r'}{r}A^t.
\end{equation*}
On the other hand, the relations 
\begin{equation*}
\mathcal{R}(\gamma _j,\gamma '_k)=(4\pi ^2R\bar{T} )_{jk},\ \mathcal{R}(\gamma '_k,\gamma _j)=(4\pi ^2\bar{R} T)_{jk}
\end{equation*}
hold. Thus, by using $\mathrm{Im}R\bar{T} =-\frac{1}{4\pi }\frac{r'}{r}A^t$ and $\mathrm{Im}\bar{R} T=\frac{1}{4\pi }\frac{r'}{r}A^t$, we obtain 
\begin{equation*}
\mathrm{Im}\mathcal{R}(\gamma _j,\gamma' _k)=-\pi \frac{r'}{r}a_{kj},\ \mathrm{Im}\mathcal{R}(\gamma '_k,\gamma _j)=\pi \frac{r'}{r}a_{kj}.
\end{equation*}
\end{proof}
Now, we consider a projectively flat bundle $\mathcal{E}_{(r,A,\mu ,\mathcal{U})}\rightarrow T^{2n}_{J=T}$ of rank $r'$ whose factor of automorphy $j : L\times \mathbb{C}^n\rightarrow GL(r';\mathbb{C})$ and connection $\tilde{\nabla} _{(r,A,\mu ,\mathcal{U})}=d+\tilde{\omega } _{(r,A,\mu ,\mathcal{U})}$ are expressed locally as follows :
\begin{align*}
&j(\gamma ,z)=U(\gamma )\mathrm{exp} \left\{ \frac{1}{r'}\mathcal{R}(z,\gamma )+\frac{1}{2r'}\mathcal{R}(\gamma ,\gamma ) \right\},\\
&\tilde{\omega } _{(r,A,\mu ,\mathcal{U})}=-\frac{1}{r'}dz^t R \bar{z} \cdot I_{r'}+\frac{\mathbf{i}}{2\pi r}\bar{\mu }^t(T-\bar{T})^{-1}dz\cdot I_{r'}-\frac{\mathbf{i}}{2\pi r} \mu ^t (T-\bar{T} )^{-1}dz\cdot I_{r'}.
\end{align*}
Here, $U(\gamma _j),U(\gamma '_k)\in U(r')$ ($j,k=1,\cdots,n$) satisfy the relations 
\begin{align}
&U(\gamma _j)U(\gamma _k)=U(\gamma _k)U(\gamma _j), \label{cc1} \\ 
&U(\gamma '_j)U(\gamma '_k)=U(\gamma '_k)U(\gamma '_j), \label{cc2} \\
&\zeta ^{-a_{kj}}U(\gamma '_k)U(\gamma _j)=U(\gamma _j)U(\gamma '_k). \label{cc3}
\end{align}
Note that these relations are equivalent to the cocycle condition of $E_{(r,A,\mu ,\mathcal{U})}$. Therefore, we can denote
\begin{equation*}
\mathcal{U}=\Bigl \{ U(\gamma _j), U(\gamma '_k)\in U(r') \ | \ (\ref{cc1}), \ (\ref{cc2}), \ (\ref{cc3}), \ j,k=1,\cdots, n \Bigr \}.
\end{equation*}
The purpose of this section is to interpret holomorphic vector bundles $E_{(r,A,\mu,\mathcal{U})}$ in the language of factors of automorphy, namely, to prove $E_{(r,A,\mu ,\mathcal{U})}\cong \mathcal{E}_{(r,A,\mu ,\mathcal{U})}$ (Theorem \ref{theo3.6}). It is clear that the curvature form $\tilde{\Omega }_{(r,A,\mu ,\mathcal{U})}$ of $\mathcal{E}_{(r,A,\mu ,\mathcal{U})}$ is expressed locally as
\begin{equation*}
\tilde{\Omega } _{(r,A,\mu ,\mathcal{U})}=\frac{1}{r'}dz^t R d\bar{z} \cdot I_{r'}.
\end{equation*}
Hence, we fix $r$, $A$, $\mu $ (note that $r'$ is uniquely defined by using $r$ and $A$), and by comparing the definition of $E_{(r,A,\mu ,\mathcal{U})}$ with the definition of $\mathcal{E}_{(r,A,\mu ,\mathcal{U})}$, we see that the cardinality of the set $\{E_{(r,A,\mu ,\mathcal{U})}\}$ is equal to the cardinality of the set $\{\mathcal{E}_{(r,A,\mu ,\mathcal{U})}\}$. Thus, we expect that there exists an isomorphism $\Psi : E_{(r,A,\mu ,\mathcal{U})} \stackrel{\sim}{\rightarrow} \mathcal{E}_{(r,A,\mu ,\mathcal{U})}$ which gives a correspondence between $\{ E_{(r,A,\mu ,\mathcal{U})} \}$ and $\{ \mathcal{E}_{(r,A,\mu ,\mathcal{U})} \}$. Actually, the following theorem holds, and this is the main theorem in this section.
\begin{theo} \label{theo3.6}
One has $E_{(r,A,\mu ,\mathcal{U})}\cong \mathcal{E}_{(r,A,\mu ,\mathcal{U})}$, where an isomorphism $\Psi : E_{(r,A,\mu ,\mathcal{U})} \stackrel{\sim }{\rightarrow } \mathcal{E}_{(r,A,\mu ,\mathcal{U})}$ is expressed locally as 
\begin{align*}
\Psi (z,\bar{z} )=&\mathrm{exp} \biggl\{ \frac{\mathbf{i}}{4\pi r'}z^t \mathcal{A} z+\frac{\mathbf{i}}{4\pi r'}\bar{z} ^t \bar{\mathcal{A}} \bar{z} -\frac{\mathbf{i}}{2\pi r'}z^t \mathcal{A}\bar{z} -\frac{\mathbf{i}}{2\pi r}z^t \{ (T-\bar{T})^{-1} \}^t \bar{\mu } \\
&+\frac{\mathbf{i}}{2\pi r}\bar{z} ^t \{ (T-\bar{T} )^{-1} \}^t \mu \biggr\}\cdot I_{r'},
\end{align*}
\hspace{4.4mm} $\mathcal{A}:=\frac{r'}{r}\{ (T-\bar{T} )^{-1} \}^t \bar{T} ^t A^t (T-\bar{T} )^{-1}$.
\end{theo}
\begin{proof}
Note that $\mathcal{A}$ is a symmetric matrix because $AT=(AT)^t$. We construct an isomorphism $\Psi : E_{(r,A,\mu ,\mathcal{U})}\rightarrow \mathcal{E}_{(r,A,\mu ,\mathcal{U})}$ explicitly such that its local expression is 
\begin{equation*}
\Psi (z,\bar{z} )=\psi (z,\bar{z} )\cdot I_{r'},
\end{equation*}
where $\psi (z,\bar{z} )$ is a function defined locally. By solving the differential equation
\begin{equation*}
\tilde{\nabla} _{(r,A,\mu ,\mathcal{U})}\Psi (z,\bar{z} )=\Psi (z,\bar{z} )\nabla_{(r,A,\mu ,\mathcal{U})}, 
\end{equation*}
we obtain the solution
\begin{align*}
\psi (z,\bar{z} )=&c\cdot \mathrm{exp} \biggl\{ \frac{\mathbf{i}}{4\pi r'}z^t \mathcal{A} z+\frac{\mathbf{i}}{4\pi r'}\bar{z} ^t \bar{\mathcal{A}} \bar{z} -\frac{\mathbf{i}}{2\pi r'}z^t \mathcal{A}\bar{z} -\frac{\mathbf{i}}{2\pi r}z^t\{ (T-\bar{T})^{-1} \}^t \bar{\mu } \\
&+\frac{\mathbf{i}}{2\pi r}\bar{z} ^t \{ (T-\bar{T} )^{-1} \}^t \mu \biggr\},
\end{align*}
where $c$ is an arbitrary constant, so by setting $c=1$, we have
\begin{align*}
\Psi (z,\bar{z} )=&\mathrm{exp} \biggl\{ \frac{\mathbf{i}}{4\pi r'}z^t \mathcal{A} z+\frac{\mathbf{i}}{4\pi r'}\bar{z} ^t \bar{\mathcal{A}} \bar{z} -\frac{\mathbf{i}}{2\pi r'}z^t \mathcal{A}\bar{z} -\frac{\mathbf{i}}{2\pi r}z^t \{ (T-\bar{T})^{-1} \}^t \bar{\mu } \\
&+\frac{\mathbf{i}}{2\pi r}\bar{z} ^t \{ (T-\bar{T} )^{-1} \}^t \mu \biggr\}\cdot I_{r'}.
\end{align*}
By using this $\Psi : E_{(r,A,\mu ,\mathcal{U})}\rightarrow \mathcal{E}_{(r,A,\mu ,\mathcal{U})}$, we transform the transition functions of $E_{(r,A,\mu ,\mathcal{U})}$. We can verify that the relation
\begin{equation}
\frac{\mathbf{i}}{2\pi }(\bar{\mathcal{A}} -\mathcal{A})=R \label{t2}
\end{equation}
holds as follows. By a direct calculation, we see
\begin{align*}
&\frac{\mathbf{i}}{2\pi }\bar{\mathcal{A}} -R\\
&=\frac{\mathbf{i}}{2\pi }\frac{r'}{r}\{ (T-\bar{T} )^{-1} \}^t T^t A^t (T-\bar{T} )^{-1}-\frac{\mathbf{i}}{2\pi } \frac{r'}{r}\{ (T-\bar{T} )^{-1} \}^t (T-\bar{T} )^t A^t (T-\bar{T} )^{-1} \\
&=\frac{\mathbf{i}}{2\pi }\frac{r'}{r}\{ (T-\bar{T} )^{-1} \}^t \bar{T}^t A^t (T-\bar{T} )^{-1}\\     
&=\frac{\mathbf{i}}{2\pi }\mathcal{A},                                              
\end{align*}
so one has 
\begin{equation*}
\frac{\mathbf{i}}{2\pi }(\bar{\mathcal{A}} -\mathcal{A})=R. 
\end{equation*}
Since we can regard
\begin{equation*}
e^{\frac{\mathbf{i}}{r}a_j y}=e^{\frac{\mathbf{i}}{r}a_j (T-\bar{T} )^{-1} z-\frac{\mathbf{i}}{r}a_j (T-\bar{T} )^{-1} \bar{z}},  
\end{equation*}
we calculate the formula
\begin{equation}
\Bigl( \Psi (z+\gamma _j ,\bar{z}+\gamma _j ) \Bigr) \left( e^{\frac{\mathbf{i}}{r}a_j (T-\bar{T} )^{-1} z-\frac{\mathbf{i}}{r}a_j (T-\bar{T} )^{-1} \bar{z}}V_j \right) \Bigl( \Psi ^{-1}(z,\bar{z} ) \Bigr), \label{t3}
\end{equation}
where $j=1,\cdots,n$. We set 
\begin{equation*}
(\bar{\mathcal{A}} -\mathcal{A})_j :=(\bar{\mathcal{A}} _{1j}-\mathcal{A}_{1j},\cdots,\bar{\mathcal{A}} _{nj}-\mathcal{A}_{nj}),\ R_j:=(R_{1j},\cdots, R_{nj}).
\end{equation*}
By using the identity (\ref{t2}), the formula (\ref{t3}) turns out to be
\begin{align*}
&\mathrm{exp} \bigg\{ \frac{\mathbf{i}}{r'}(\bar{\mathcal{A}} -\mathcal{A})_j \bar{z} +\frac{\pi \mathbf{i}}{r'}(\bar{\mathcal{A}}-\mathcal{A})_{jj} -\frac{\mathbf{i}}{r}(\{ (T-\bar{T})^{-1} \}^t \bar{\mu } )_j +\frac{\mathbf{i}}{r}(\{ (T-\bar{T} )^{-1} \}^t \mu )_j \\
&+\frac{\mathbf{i}}{r}a_j (T-\bar{T})^{-1} z-\frac{\mathbf{i}}{r}a_j (T-\bar{T} )^{-1}\bar{z} \bigg\} V_j \\
&=\mathrm{exp} \left\{ \frac{2\pi }{r'}R_j z+\frac{2\pi ^2}{r'}R_{jj}+\frac{\mathbf{i}}{r}q_j \right\} V_j \\
&=\mathrm{exp} \left\{ \frac{\mathbf{i}}{r}q_j \right\}V_j \ \mathrm{exp} \left\{\frac{1}{r'}\mathcal{R}(z,\gamma _j)+\frac{1}{2r'}\mathcal{R}(\gamma _j ,\gamma _j) \right\}.
\end{align*}
In particular, $\mathrm{exp} \left\{ \frac{\mathbf{i}}{r}q_j \right\}$ is a purely imaginary number, and this fact indicates
\begin{equation*}
U(\gamma_j):=\mathrm{exp} \left\{ \frac{\mathbf{i}}{r}q_j \right\}V_j \in U(r').
\end{equation*}
Similarly, we also calculate the formula
\begin{equation}
\Bigl(\Psi (z+\gamma '_k,\bar{z} +\bar{\gamma '_k} )\Bigr) \Bigl( U_k \Bigr) \Bigl( \Psi ^{-1}(z,\bar{z} ) \Bigr), \label{t4}
\end{equation}
where $k=1,\cdots,n$. In order to calculate the formula (\ref{t4}), we prove the relations
\begin{gather}
\mathcal{A}T=\overline{\mathcal{A}T}, \label{t5} \\
\mathcal{A}(T-\bar{T} )=-2\pi \mathbf{i}R\bar{T}. \label{t6}
\end{gather}
We can show the identity (\ref{t5}) as follows. For $T=X+\mathbf{i}Y$, 
\begin{align*}
\mathcal{A}T&=\frac{r'}{r}\{ (T-\bar{T} )^{-1} \}^t \bar{T}^t A^t (T-\bar{T} )^{-1} T\\
                  &=-\frac{1}{4}\frac{r'}{r}(Y^{-1})^t X^t A^t Y^{-1} X-\frac{1}{4}\frac{r'}{r}A^t -\frac{\mathbf{i}}{4}\frac{r'}{r}((Y^{-1})^t X^t A^t -A^t Y^{-1} X),
\end{align*}
and since $AT=(AT)^t$ holds,
\begin{equation*}
\mathrm{Im}\mathcal{A}T=-\frac{1}{4}\frac{r'}{r}((Y^{-1})^t X^t A^t -A^t Y^{-1}X)=O.
\end{equation*}
This implies $\mathcal{A}T=\overline{\mathcal{A}T}$. Furthermore, by a direct calculation,
\begin{align*}
\mathcal{A}(T-\bar{T} )&=\frac{r'}{r}\{ (T-\bar{T} )^{-1} \}^t \bar{T}^t A^t (T-\bar{T} )^{-1}(T-\bar{T} )\\
                                &=\frac{r'}{r}\{ (T-\bar{T} )^{-1} \}^t A\bar{T} \\
                                &=-2\pi \mathbf{i}R\bar{T},
\end{align*}
so we obtain the identity (\ref{t6}). Now, we calculate the formula (\ref{t4}). We set 
\begin{gather*} 
(\mathcal{A}(T-\bar{T} ))_k :=((\mathcal{A}(T-\bar{T} ))_{1k},\cdots, (\mathcal{A}(T-\bar{T} ))_{nk} ),\\
(R\bar{T})_k :=((R\bar{T} )_{1k},\cdots, (R\bar{T} )_{nk}).
\end{gather*}
By using the identities (\ref{t5}), (\ref{t6}), the formula (\ref{t4}) turns out to be
\begin{align*}
&\mathrm{exp} \bigg\{ \frac{\mathbf{i}}{r'}(\mathcal{A}(T-\bar{T} ))_k z+\frac{\pi \mathbf{i}}{r'}((\mathcal{A}T)^t T)_{kk}+\frac{\pi \mathbf{i}}{r'}((\overline{\mathcal{A}T})^t \bar{T})_{kk}-\frac{2\pi \mathbf{i}}{r'}((\mathcal{A}T)^t \bar{T} )_{kk}\\
&-\frac{\mathbf{i}}{r}( \bar{\mu }^t (T-\bar{T})^{-1} T)_k +\frac{\mathbf{i}}{r}(\mu ^t (T-\bar{T} )^{-1} \bar{T} )_k \bigg\} U_k \\
&=\mathrm{exp} \left\{ \frac{\mathbf{i}}{r'}(\mathcal{A}(T-\bar{T} ))_k z+\frac{\pi \mathbf{i}}{r'}(T^t \mathcal{A}(T-\bar{T} ))_{kk} -\frac{\mathbf{i}}{r}p_k \right\} U_k \\
&=\mathrm{exp} \left\{ \frac{2\pi }{r'}(R\bar{T} )_k z+\frac{2\pi ^2}{r'}(T^t R\bar{T} )_{kk} -\frac{\mathbf{i}}{r}p_k \right\} U_k \\
&=\mathrm{exp} \left\{ \frac{2\pi }{r'}(R\bar{T} )_k z+\frac{2\pi ^2}{r'}(\bar{T}^t RT )_{kk} -\frac{\mathbf{i}}{r}p_k \right\} U_k \\
&=\mathrm{exp} \left\{ -\frac{\mathbf{i}}{r}p_k \right\}U_k \ \mathrm{exp} \left\{ \frac{1}{r'}\mathcal{R}(z,\gamma '_k)+\frac{1}{2r'}\mathcal{R}(\gamma '_k,\gamma '_k) \right\}.
\end{align*}
In particular, $\mathrm{exp} \left\{ -\frac{\mathbf{i}}{r}p_k \right\}$ is a purely imaginary number, and this fact indicates
\begin{equation*}
U(\gamma_k'):=\mathrm{exp} \left\{ -\frac{\mathbf{i}}{r}p_k \right\}U_k \in U(r').
\end{equation*}
Here, we remark that the matrices $U(\gamma_j)$, $U(\gamma_k')$ ($j$, $k=1,\cdots, n$) satisfy the relations (\ref{cc1}), (\ref{cc2}) and (\ref{cc3}) if and only if the matrices $V_j$, $U_k$ ($j$, $k=1,\cdots, n$) satisfy the cocycle condition
\begin{equation*}
V_jV_k=V_kV_j, \ U_jU_k=U_kU_j, \ \zeta^{-a_{kj}}U_kV_j=V_jU_k
\end{equation*}
of $E_{(r,A,\mu,\mathcal{U})}$. This completes the proof.
\end{proof}

\section{Exact triangles consisting of projectively flat bundles on $T^{2n}_{J=T}$}
The purpose of this section is to prove Theorem \ref{theo4.1} which plays an important role in section 5. In Theorem \ref{theo4.1}, we focus on exact triangles consisting of three simple projectively flat bundles $E_{(r,A,\mu,\mathcal{U})}$, $E_{(s,B,\nu,\mathcal{V})}$, $E_{(t,C,\eta,\mathcal{W})}\rightarrow T^{2n}_{J=T}$ and their shifts. We also give a geometric interpretation for such exact triangles by focusing on the dimension of intersections of the corresponding affine Lagrangian submanifolds in the last of this section.

Let us consider an exact triangle
\begin{align}
\begin{CD}
\cdots &@>>> E_{(r,A,\mu,\mathcal{U})} @>>> C(\psi) @>>> E_{(s,B,\nu,\mathcal{V})} \\
&@>\psi>> E_{(r,A,\mu,\mathcal{U})}[1] @>>> \cdots \label{triangle4.1}
\end{CD}
\end{align}
in $Tr(DG_{T^{2n}_{J=T}})$, where $\psi \in \mathrm{Ext}^1(E_{(s,B,\nu,\mathcal{V})}, E_{(r,A,\mu,\mathcal{U})})$ is a non-trivial morphism. We set
\begin{equation*}
\alpha :=\frac{1}{r}A-\frac{1}{s}B.
\end{equation*}
Since the non-triviality of $\psi \in \mathrm{Ext}^1(E_{(s,B,\nu,\mathcal{V})}, E_{(r,A,\mu,\mathcal{U})})$ implies the existence of an isomorphism $E_{(r,A,\mu,\mathcal{U})}\cong E_{(s,B,\nu,\mathcal{V})}$, i.e., $C(\psi)\cong 0$ in the case $\alpha =O$, we consider the case $\alpha \not =O$ only throughout this paper. Here, we give the following theorem. 
\begin{theo} \label{theo4.1}
In the exact triangle $(\ref{triangle4.1})$, we assume that there exists a holomorphic vector bundle $E_{(t,C,\eta, \mathcal{W})}$ such that $C(\psi)\cong E_{(t,C,\eta,\mathcal{W})}$. Then, we have
\begin{equation*}
\mathrm{rank}\hspace{0.5mm}\alpha =1.
\end{equation*}
\end{theo}
\begin{proof}
We define the 2-forms $\Omega '_{(r,A,\mu ,\mathcal{U})}$, $\Omega '_{(s,B,\nu ,\mathcal{V})}$, $\Omega '_{(t,C,\eta ,\mathcal{W})}$ by
\begin{align*}
&\Omega '_{(r,A,\mu ,\mathcal{U})}:=\frac{1}{4\pi ^2 r}dx^t A^t dy,\\
&\Omega '_{(s,B,\nu ,\mathcal{V})}:=\frac{1}{4\pi ^2 s}dx^t B^t dy,\\
&\Omega '_{(t,C,\eta ,\mathcal{W})}:=\frac{1}{4\pi ^2 t}dx^t C^t dy,
\end{align*}
respectively, namely,
\begin{align*}
&-\frac{1}{2\pi \mathbf{i}}\Omega _{(r,A,\mu ,\mathcal{U})}=\Omega '_{(r,A,\mu ,\mathcal{U})}\cdot I_{r'},\\
&-\frac{1}{2\pi \mathbf{i}}\Omega _{(s,B,\nu ,\mathcal{V})}=\Omega '_{(s,B,\nu ,\mathcal{V})}\cdot I_{s'},\\
&-\frac{1}{2\pi \mathbf{i}}\Omega _{(t,C,\eta ,\mathcal{W})}=\Omega '_{(t,C,\eta ,\mathcal{W})}\cdot I_{t'}.
\end{align*}
Since we assume $C(\psi )\cong E_{(t,C,\eta ,\mathcal{W})}$, one has $ch_i (C(\psi ))=ch_i (E_{(t,C,\eta ,\mathcal{W})})$, where $i=1,\cdots,n$ and $ch_i (E)$ denotes the $i$-th Chern character of a vector bundle $E$. In particular, 
\begin{equation*}
ch_i (C(\psi ))=ch_i (E_{(r,A,\mu ,\mathcal{U})})+ch_i (E_{(s,B,\nu ,\mathcal{V})}),
\end{equation*}
so $ch_i (C(\psi ))=ch_i (E_{(t,C,\eta ,\mathcal{W})})$ is equivalent to
\begin{equation}
ch_i (E_{(r,A,\mu ,\mathcal{U})})+ch_i (E_{(s,B,\nu ,\mathcal{V})})=ch_i (E_{(t,C,\eta ,\mathcal{W})}). \label{ch}
\end{equation}
Now we calculate $ch_i (C(\psi ))$, $ch_i (E_{(t,C,\eta ,\mathcal{W})})$ and consider the equality (\ref{ch}). It is clear that the equality (\ref{ch}) in the cases $i=0,1$ are equivalent to 
\begin{align}
&r'+s'=t', \label{c0} \\
&r'\Omega '_{(r,A,\mu ,\mathcal{U})}+s'\Omega '_{(s,B,\nu ,\mathcal{V})}=t'\Omega '_{(t,C,\eta ,\mathcal{W})}, \label{c1}
\end{align}
respectively. We consider the equality (\ref{ch}) in the case $i=2$. By a direct calculation, the equality (\ref{ch}) turns out to be
\begin{equation}
\frac{r'}{2}(\Omega '_{(r,A,\mu ,\mathcal{U})})^2 +\frac{s'}{2}(\Omega '_{(s,B,\nu ,\mathcal{V})})^2 =\frac{t'}{2}(\Omega '_{(t,C,\eta ,\mathcal{W})})^2, \label{c2}
\end{equation}
and we obtain the relation 
\begin{equation}
(r't'-r'^2)(\Omega '_{(r,A,\mu ,\mathcal{U})})^2 +(s't'-s'^2)(\Omega '_{(s,B,\nu ,\mathcal{V})})^2 =2r's'\Omega '_{(r,A,\mu ,\mathcal{U})}\wedge \Omega '_{(s,B,\nu ,\mathcal{V})} \label{c2'}
\end{equation}
by substituting the equality (\ref{c1}) into the equality (\ref{c2}). Furthermore, by substituting the equality (\ref{c0}) into the equality (\ref{c2'}), the equality (\ref{c2'}) turns out to be
\begin{equation*}
r's'(\Omega '_{(r,A,\mu ,\mathcal{U})})^2 +r's'(\Omega '_{(s,B,\nu ,\mathcal{V})})^2 =2r's'\Omega '_{(r,A,\mu ,\mathcal{U})}\wedge \Omega '_{(s,B,\nu ,\mathcal{V})},
\end{equation*}
and this relation is equivalent to
\begin{equation}
(\Omega '_{(r,A,\mu ,\mathcal{U})}-\Omega '_{(s,B,\nu ,\mathcal{V})})^2 =0. \label{c2''}
\end{equation}
In general, for $i\geq 3$, we obtain the equality 
\begin{align}
&\left( r'\sum_{k=1}^{i-1} \binom{i-1}{k} r'^{i-1-k}s'^k \right) \Biggl( \Omega '_{(r,A,\mu ,\mathcal{U})} \Biggr) ^i \notag \\
&+\left( s'\sum_{k=0}^{i-2} \binom{i-1}{k} r'^{i-1-k}s'^k \right) \Biggl( \Omega '_{(s,B,\nu ,\mathcal{V})} \Biggr) ^i \notag \\
&-\sum_{k=1}^{i-1} \binom{i}{k} \Biggl( r'\Omega '_{(r,A,\mu ,\mathcal{U})} \Biggr)^{i-k} \Biggl( s'\Omega '_{(s,B,\nu ,\mathcal{V})} \Biggr) ^k =0 \label{ci}
\end{align}
by expanding the equality (\ref{ch}). Note that the left hand side of the equality (\ref{ci}) can be factored as
\begin{align*}
&(\Omega '_{(r,A,\mu ,\mathcal{U})}-\Omega '_{(s,B,\nu ,\mathcal{V})})^2 \\
&\times \sum_{l=0}^{i-2} \Bigg\{ \Biggl( \sum_{k=1}^l (i-l-1) \binom{i-1}{k-1} r'^{i-k}s'^k +(l+1)\sum_{k=l+1}^{i-1} \binom{i-1}{k} r'^{i-k} s'^k \Biggr) \\
&\times ( \Omega '_{(r,A,\mu ,\mathcal{U})} ) ^{i-l-2} ( \Omega '_{(s,B,\nu ,\mathcal{V})} ) ^l \Bigg\}.
\end{align*}
Hence, when the equality (\ref{c2''}) holds, the equality (\ref{ci}) holds automatically. Moreover, by the definition of $\Omega '_{(r,A,\mu ,\mathcal{U})}$ and $\Omega '_{(s,B,\nu ,\mathcal{V})}$,
\begin{equation*}
\Omega '_{(r,A,\mu ,\mathcal{U})}-\Omega '_{(s,B,\nu ,\mathcal{V})}=\frac{1}{4\pi ^2} dx^t \left( \frac{1}{r}A^t -\frac{1}{s}B^t \right) dy=\frac{1}{4\pi ^2} dx^t \alpha ^t dy,
\end{equation*}
so by a direct calculation, one has 
\begin{align*}
&(\Omega '_{(r,A,\mu ,\mathcal{U})}-\Omega '_{(s,B,\nu ,\mathcal{V})})^2 \\
&=\frac{1}{8\pi ^4}\sum_{1\leq i<j\leq n,1\leq k<l\leq n}(\alpha _{ik}\alpha _{jl}-\alpha _{il}\alpha _{jk})dx_k \wedge dy_i \wedge dx_l \wedge dy_j.
\end{align*}
Thus, the equality (\ref{c2''}) is equivalent to 
\begin{equation}
\mathrm{det}\left( \begin{array}{ccc} \alpha _{ik} & \alpha _{il} \\ \alpha _{jk} & \alpha _{jl} \end{array} \right) = \alpha _{ik} \alpha _{jl}-\alpha _{il} \alpha _{jk}=0, \label{c2'''}
\end{equation}
where $1\leq i<j\leq n$, $1\leq k<l\leq n$. 

Now, in order to prove the statement of this theorem, we apply elementary row operations to the matrix $\alpha $. Since we assume $\alpha \not= O$, there exists an $\alpha_{ij}\not =0$. First, we multiply the first row of $\alpha $ by $\alpha _{ij}$ :
\begin{equation*}
\alpha \longrightarrow 
\alpha ':=
\left( \begin{array}{@{\,}cccccc@{\,}}
\alpha _{11}\alpha _{ij} & \ldots & \alpha _{1j}\alpha _{ij} & \ldots & \alpha _{1n}\alpha _{ij} \\
\vdots                     & \ddots & \vdots                     & \ddots & \vdots                      \\
\alpha _{i1}                & \ldots & \alpha _{ij}                & \ldots & \alpha _{in}                \\
\vdots                     & \ddots & \vdots                     & \ddots & \vdots                    \\
\alpha _{n1}                & \ldots & \alpha _{nj}               & \ldots & \alpha _{nn} \end{array} \right).
\end{equation*}
Next, we add the $i$-th row of $\alpha '$ multiplied by $-\alpha _{1j}$ to the first row of $\alpha '$ :
\begin{equation*}
\alpha ' \longrightarrow 
\alpha '':=
\left( \begin{array}{@{\,}cccccc@{\,}}
\alpha _{11}\alpha _{ij}-\alpha _{i1}\alpha _{1j} & \ldots & 0 & \ldots & \alpha _{1n}\alpha _{ij}-\alpha _{in}\alpha _{1j} \\
\vdots                     & \ddots & \vdots                     & \ddots & \vdots                     \\
\alpha _{i1}                & \ldots & \alpha _{ij}                & \ldots & \alpha _{in}                \\
\vdots                     & \ddots & \vdots                     & \ddots & \vdots                    \\
\alpha _{n1}                & \ldots & \alpha _{nj}               & \ldots & \alpha _{nn} \end{array} \right).
\end{equation*}
Then, by using the equality (\ref{c2'''}), we see that all components of the first row of $\alpha ''$ are zero, namely,
\begin{equation*}
\alpha ''=
\left( \begin{array}{@{\,}cccccc@{\,}}
0 & \ldots & 0 & \ldots & 0 \\
\vdots                     & \ddots & \vdots                     & \ddots & \vdots                     \\
\alpha _{i1}                & \ldots & \alpha _{ij}                & \ldots & \alpha _{in}                \\
\vdots                     & \ddots & \vdots                     & \ddots & \vdots                     \\
\alpha _{n1}                & \ldots & \alpha _{nj}               & \ldots & \alpha _{nn} \end{array} \right).
\end{equation*}
By applying elementary row operations to $\alpha ''$ similarly as above, $\alpha ''$ is transformed as follows finally :
\begin{equation*}
\left( \begin{array}{@{\,}cccccc@{\,}}
0 & \ldots & 0 & \ldots & 0 \\
\vdots                     & \ddots & \vdots                     & \ddots & \vdots                     \\
\alpha _{i1}                & \ldots & \alpha _{ij}                & \ldots & \alpha _{in}                \\
\vdots                     & \ddots & \vdots                     & \ddots & \vdots                     \\
0 & \ldots & 0 & \ldots & 0 \end{array} \right).
\end{equation*}
Thus, we can conclude that the relation
\begin{equation*}
\mathrm{rank}\hspace{0.5mm}\alpha =1
\end{equation*}
holds.
\end{proof}
We give a geometric interpretation for Theorem \ref{theo4.1} from the viewpoint of the homological mirror symmetry for $(T^{2n}_{J=T}, \check{T}^{2n}_{J=T})$. By using suitable parameters $\check{p}$, $\check{u}\in \mathbb{R}^n$, we can express the affine Lagrangian submanifolds which correspond to holomorphic vector bundles $E_{(r,A,\mu,\mathcal{U})}$, $E_{(s,B,\nu,\mathcal{V})}$ as $L_{(r,A,\check{p})}$, $L_{(s,B,\check{u})}$, respectively (cf. \cite[Theorem 5.1]{kazushi3}, Theorem \ref{bijection1}). Note that the non-triviality of
\begin{equation*}
\mathrm{Ext}^1(E_{(s,B,\nu,\mathcal{V})}, E_{(r,A,\mu,\mathcal{U})})
\end{equation*}
implies
\begin{equation*}
L_{(r,A,\check{p})}\cap L_{(s,B,\check{u})}\not =\emptyset
\end{equation*}
in the description of the homological mirror symmetry (see also the relations (\ref{munu}), (\ref{nueta}), and p.32, p.33 in \cite{D}). Then, the relation $\mathrm{rank}\hspace{0.5mm}\alpha =1$ in Theorem \ref{theo4.1} indicates 
\begin{equation*}
\mathrm{codim}(L_{(r,A,\check{p})}\cap L_{(s,B,\check{u})})=1.
\end{equation*}
For example, let us consider the case $n=1$, i.e., the case of elliptic curves $(T^2_{J=T}, \check{T}^2_{J=T})$. We focus on the exact triangle
\begin{align*}
\begin{CD}
\cdots &@>>> E_{(1,0,\mu,\mathcal{U})} @>>> C(\psi) @>>> E_{(1,1,\nu,\mathcal{V})} \\
&@>\psi>> E_{(1,0,\mu,\mathcal{U})}[1] @>>> \cdots 
\end{CD}
\end{align*}
in $Tr(DG_{T^2_{J=T}})$, where
\begin{equation*}
\mathcal{U}=\mathcal{V}=\Bigl\{ V_1 =U_1 =1 \in U(1) \Bigr\} .
\end{equation*}
Then, by \cite[Theorem 4.10]{kazushi}, $C(\psi )\cong E_{(2,1,\eta ,\mathcal{W})}$ if and only if $\eta \equiv \mu +\nu +\pi +\pi T\ (\mathrm{mod} \ 2\pi (\mathbb{Z}\oplus T\mathbb{Z}))$, where
\begin{equation*}
\mathcal{W}=\left\{ V_1 =\left( \begin{array}{ccc} 0&1\\1&0 \end{array} \right), U_1 =\left( \begin{array}{ccc} 1&0 \\ 0&-1 \end{array} \right) \in U(2) \right\},
\end{equation*}
and we can actually check that $\mathrm{codim}(L_{(1,0,\check{p})}\cap L_{(1,1,\check{u})})=1$ holds in this case. 

\section{An application}
In \cite{D}, as an application of Theorem \ref{theo4.1}, we studied exact triangles consisting of three simple projectively flat bundles $E_{(r,A,\mu,\mathcal{U})}$, $E_{(s,B,\nu,\mathcal{V})}$, $E_{(t,C,\eta,\mathcal{W})}\rightarrow T^{2n}_{J=T}$ ($n\geq 2$) and their shifts with the assumption $\mathrm{rank}\hspace{0.5mm}E_{(r,A,\mu,\mathcal{U})}=1$. In particular, the main result is given in \cite[Theorem 5.6]{D}. The purpose of this section is to extend \cite[Theorem 5.6]{D} to general settings (Theorem \ref{mainth}).

Also, in this section, we sometimes consider the holomorphic vector bundle of the form $E_{(r,aE_{ij},\mu,\mathcal{U})}\rightarrow T^{2n}_{J=T}$ of rank $r'$, where $r\in \mathbb{N}$, $a\in \mathbb{Z}$, and $\mu \in \mathbb{R}^n \oplus T^t \mathbb{R}^n$. When we consider such a holomorphic vector bundle, we can take the set 
\begin{equation*}
\mathcal{U}_{r'}:=\left\{ V_j=V, \ V_l=I_{r'}, \ U_i=U^{-\frac{r'}{r}a}, \ U_k=I_{r'} \in U(r') \ | \ l, k=1,\cdots, n, \ l\not =j, \ k\not =i \right\}
\end{equation*}
as an example of $\mathcal{U}$, where
\begin{equation*}
V:=\left( \begin{array}{cccc} 0\ \ 1&&\\&\ddots\ddots&\\&&1\\1&&0 \end{array} \right), \ U:=\left( \begin{array}{cccc} 1&&&\\&\zeta '&&\\&&\ddots&\\&&&(\zeta ')^{r'-1} \end{array} \right)\in U(r'), \ \zeta ':=e^{\frac{2\pi \mathbf{i}}{r'}}.
\end{equation*}
Throughout this section, we use the notation $\mathcal{U}_{r'}$ in this sense. Note that
\begin{equation*}
\mathcal{U}_1=\Bigl\{ V_1=\cdots =V_n=U_1=\cdots U_n=1 \in U(1) \Bigr\}.
\end{equation*}

\subsection{Previous work}
In this subsection, we recall the discussions in subsection 5.2 in \cite{D}. Roughly speaking, in subsection 5.2 in \cite{D}, we proved that an exact triangle consisting of three simple projectively flat bundles $E_{(1,A,\mu,\mathcal{U})}$, $E_{(s,B,\nu,\mathcal{V})}$, $E_{(t,C,\eta,\mathcal{W})}\rightarrow T^{2n}_{J=T}$ ($n\geq 2$) and their shifts is obtained as the pullback of an exact triangle defined on a suitable one-dimensional complex torus (\cite[Theorem 5.6]{D}).

Let us consider an exact triangle 
\begin{align}
\begin{CD}
\cdots &@>>> E_{(r,A,\mu,\mathcal{U})} @>>> C(\psi) @>>> E_{(s,B,\nu,\mathcal{V})} \\
&@>\psi>> E_{(r,A,\mu,\mathcal{U})}[1] @>>> \cdots \label{tr1}
\end{CD}
\end{align}
in $Tr(DG_{T^{2n}_{J=T}})$. In the definition of the holomorphic vector bundle of the form $E_{(r,A,\mu,\mathcal{U})}$, $r'\in \mathbb{N}$ is the rank of $E_{(r,A,\mu,\mathcal{U})}$, and $\frac{r'}{r}A\in M(n;\mathbb{Z})$ corresponds to the 1-st Chern class of $E_{(r,A,\mu,\mathcal{U})}$. In particular, when we set 
\begin{equation*}
A':=\frac{r'}{r}A, \ p':=\frac{r'}{r}p, \ q':=\frac{r'}{r}q, \mu':=\frac{r'}{r}\mu,
\end{equation*}
we can regard
\begin{equation*}
E_{(r,A,\mu,\mathcal{U})}=E_{(r',A',\mu',\mathcal{U})}.
\end{equation*}
Hereafter, we use the notation $E_{(r',A',\mu',\mathcal{U})}$ instead of $E_{(r,A,\mu,\mathcal{U})}$ in order to specify the rank and the 1-st Chern class of $E_{(r,A,\mu,\mathcal{U})}$ (we will also use the notations $E_{(s',B',\nu',\mathcal{V})}$ and $E_{(t',C',\eta',\mathcal{W})}$ in this sense). As a result, we can rewrite the exact triangle (\ref{tr1}) to the exact triangle
\begin{align}
\begin{CD}
\cdots &@>>> E_{(r',A',\mu',\mathcal{U})} @>>> C(\psi) @>>> E_{(s',B',\nu',\mathcal{V})} \\
&@>\psi>> E_{(r',A',\mu',\mathcal{U})}[1] @>>> \cdots . \label{tr2}
\end{CD}
\end{align}
We assume that $r=1$, i.e., $r'=1$, and $C(\psi)\in DG_{T^{2n}_{J=T}}$, namely, assume that $E_{(r',A',\mu',\mathcal{U}')}$ is a holomorphic line bundle, and there exists a holomorphic vector bundle $E_{(t',C',\eta',\mathcal{W})}$ such that $C(\psi)\cong E_{(t',C',\eta',\mathcal{W})}$. Therefore, the exact triangle (\ref{tr2}) is equivalent to the following :
\begin{align}
\begin{CD}
\cdots &@>>> E_{(1,A',\mu',\mathcal{U})} @>>> E_{(t',C',\eta',\mathcal{W})} @>>> E_{(s',B',\nu',\mathcal{V})} \\
&@>>> E_{(1,A',\mu',\mathcal{U})}[1] @>>> \cdots . \label{tr3}
\end{CD}
\end{align}
Then, by Theorem \ref{theo4.1}, 
\begin{equation*}
\mathrm{rank}\hspace{0.5mm}\alpha=1
\end{equation*}
holds, and under the assumption $\mathrm{rank}\hspace{0.5mm}\alpha=1$, we obtain the following two propositions (see \cite[Proposition 5.1]{D} and \cite[Proposition 5.2]{D}\footnote{Although we consider the transformation of the matrix $s'\alpha $ by two matrices $\mathcal{A}$, $\mathcal{D}\in SL(n;\mathbb{Z})$ in Proposition 5.1 in this paper, the matrix $s\alpha $ is transformed by using two matrices $\mathcal{A}$, $\mathcal{D}\in SL(n;\mathbb{Z})$ in \cite[Proposition 5.1]{D}. However, we can prove Proposition 5.1 in this paper in a similar way as described in the proof of \cite[Proposition 5.1]{D}.}).
\begin{proposition} \label{d5.1}
Assume $\mathrm{rank}\hspace{0.5mm}\alpha=1$. Then, there exist two matrices $\mathcal{A}$, $\mathcal{D}\in SL(n;\mathbb{Z})$ such that 
\begin{equation*}
\mathcal{D}^t(s'\alpha)\mathcal{A}=-NE_{ij},
\end{equation*}
where $N\in \mathbb{N}$ and $E_{ij}$ denotes the matrix unit.
\end{proposition}
\begin{proposition} \label{d5.2}
We assume $\mathrm{rank}\hspace{0.5mm}\alpha=1$, and take a pair $(\mathcal{A}, \mathcal{D})$ of two matrices $\mathcal{A}$, $\mathcal{D}\in SL(n;\mathbb{Z})$ which satisfy the statement of Proposition \ref{d5.1}. Then,
\begin{equation*}
(\mathcal{A}^{-1}T\mathcal{D})_{ji'}=0 \ (1\leq i' \not =i\leq n)
\end{equation*}
and
\begin{equation*}
\mathrm{Im}(\mathcal{A}^{-1}T\mathcal{D})_{ji}\not =0
\end{equation*}
hold.
\end{proposition}
For simplicity, we set $T':=\mathcal{A}^{-1}T\mathcal{D}$. We can consider the $n$-dimensional complex torus $T^{2n}_{J=T'}$ by using this matrix $T'$. Let us denote the local complex coordinates of $T^{2n}_{J=T'}$ by $Z=X+T'Y=X+\mathcal{A}^{-1}T\mathcal{D}Y$, where
\begin{equation*}
Z:=(Z_1,\cdots, Z_n)^t, \ X:=(X_1,\cdots, X_n)^t, \ Y:=(Y_1,\cdots, Y_n)^t.
\end{equation*}
Then, two $n$-dimensional complex tori $T^{2n}_{J=T}$, $T^{2n}_{J=T'}$ are biholomorphic to each other, and the biholomorphic map $\varphi : T^{2n}_{J=T'} \stackrel{\sim}{\rightarrow} T^{2n}_{J=T}$ is actually given by
\begin{equation}
\varphi(Z)=\mathcal{A}Z. \label{phi}
\end{equation}
The biholomorphicity of the map $\varphi$ implies the holomorphicity of the pullback bundle $\varphi^*E_{(1,A',\mu',\mathcal{U})}\rightarrow T^{2n}_{J=T'}$, and we can regard $\varphi^*E_{(1,A',\mu',\mathcal{U})}=E_{(1,\tilde{A}',\tilde{\mu}',\mathcal{U}')}$, where
\begin{equation*}
\tilde{A}'=(\tilde{a}'_{kl}):=\mathcal{D}^t A'\mathcal{A}, \ \tilde{\mu}':=\mathcal{D}^t\mu',
\end{equation*}
and $\mathcal{U}'$ is defined by using the data $(\mathcal{U}, \mathcal{A}, \mathcal{D})$. We will also use the notations $E_{(s',\tilde{B}',\tilde{\nu}',\mathcal{V}')}$ and $E_{(t',\tilde{C}',\tilde{\eta}',\mathcal{W}')}$ in this sense. Hence, by the biholomorphic map $\varphi$, the exact triangle (\ref{tr3}) is transformed to the exact triangle
\begin{align}
\begin{CD}
\cdots &@>>> E_{(1,\tilde{A}',\tilde{\mu}',\mathcal{U}')} @>>> E_{(t',\tilde{C}',\tilde{\eta}',\mathcal{W}')} @>>> E_{(s',\tilde{B}',\tilde{\nu}',\mathcal{V}')} \\
&@>>> E_{(1,\tilde{A}',\tilde{\mu}',\mathcal{U}')}[1] @>>> \cdots \label{tr4}
\end{CD}
\end{align}
in $Tr(DG_{T^{2n}_{J=T'}})$. Here, we apply the triangulated functor $Tr(DG_{T^{2n}_{J=T'}})\stackrel{\sim}{\rightarrow} Tr(DG_{T^{2n}_{J=T'}})$ which is induced by the operator $\otimes E_{(1,-\tilde{A}',0,\mathcal{U}_1)}$ to the exact triangle (\ref{tr4}), so it is mapped to the following :
\begin{align}
\begin{CD}
\cdots &@>>> E_{(1,O,\tilde{\mu}',\mathcal{U}')} @>>> E_{(t',NE_{ij},\tilde{\eta}',\mathcal{W}')} @>>> E_{(s',NE_{ij},\tilde{\nu}',\mathcal{V}')} \\
&@>>> E_{(1,O,\tilde{\mu}',\mathcal{U}')}[1] @>>> \cdots . \label{tr5}
\end{CD}
\end{align}
Note that $gcd(s',N)=gcd(t',N)=1$ holds in the exact triangle (\ref{tr5}), and this fact closely related to the simplicity of the holomorphic vector bundles $E_{(s',NE_{ij},\tilde{\nu}',\mathcal{V}')}$ and $E_{(t',NE_{ij},\tilde{\eta}',\mathcal{W}')}$ (cf. \cite[Proposition 5.5]{D}). Furthermore, for two arbitrary holomorphic vector bundles $E_{(r',A',\mu',\mathcal{U})}$, $E_{(r',A',\nu',\mathcal{V})}\rightarrow T^{2n}_{J=T'}$, it is known that there exist $\eta' \in \mathbb{R}^n \oplus T'^t \mathbb{R}^n$ and the set $\mathcal{W}$ such that
\begin{equation*}
E_{(r',A',\nu',\mathcal{V})}\cong E_{(r',A',\mu',\mathcal{U})} \otimes E_{(1,O,\eta',\mathcal{W})}
\end{equation*}
holds (see \cite{matsu}, \cite{mukai}). In particular, we determine these data $(\eta', \mathcal{W})$ explicitly in \cite[Theorem 3.4]{kazushi3}. Hence, there exist
\begin{equation*}
\tilde{\mu}_0', \ \tilde{\nu}_0', \ \tilde{\eta}_0' \in \mathbb{R}^n \oplus T'^t \mathbb{R}^n
\end{equation*}
such that
\begin{equation*}
E_{(1,O,\tilde{\mu}',\mathcal{U}')}\cong E_{(1,O,\tilde{\mu}_0',\mathcal{U}_1)}, \ E_{(s',NE_{ij},\tilde{\nu}',\mathcal{V}')}\cong E_{(s',NE_{ij},\tilde{\nu}_0',\mathcal{U}_{s'})}, \ E_{(t',NE_{ij},\tilde{\eta}',\mathcal{W}')}\cong E_{(t',NE_{ij},\tilde{\eta}_0',\mathcal{U}_{t'})}
\end{equation*}
hold. As a result, we may consider the exact triangle
\begin{align}
\begin{CD}
\cdots &@>>> E_{(1,O,\tilde{\mu}_0',\mathcal{U}_1)} @>>> E_{(t',NE_{ij},\tilde{\eta}_0',\mathcal{U}_{t'})} @>>> E_{(s',NE_{ij},\tilde{\nu}_0',\mathcal{U}_{s'})} \\
&@>>> E_{(1,O,\tilde{\mu}_0',\mathcal{U}_1)}[1] @>>> \cdots \label{tr6}
\end{CD}
\end{align}
in $Tr(DG_{T^{2n}_{J=T'}})$ instead of the exact triangle (\ref{tr5}). Below, we explain the statement of \cite[Theorem 5.6]{D}. Let us define a holomorphic projection $\pi : T^{2n}_{J=T'}\rightarrow T^2_{J=t_{ji}'}=\mathbb{C}/2\pi(\mathbb{Z}\oplus t_{ji}'\mathbb{Z})$ by 
\begin{equation*}
\pi (Z)=Z_j=X_j+t_{ji}'Y_i.
\end{equation*}
Moreover, we decompose $\tilde{\mu}_0'=\tilde{p}_0'+T'^t \tilde{q}_0'$, and define
\begin{equation*}
(\tilde{\mu}_0')_i :=(\tilde{p}_0')_i+t_{ji}'(\tilde{q}_0')_j \in \mathbb{R}\oplus t_{ji}'\mathbb{R}.
\end{equation*}
Similarly as in the case of $(\tilde{\mu}_0')_i$, for two parameters $\tilde{\nu}_0'$, $\tilde{\eta}_0'$, we can associate the notations
\begin{equation*}
(\tilde{\nu}_0')_i, \ (\tilde{\eta}_0')_i\in \mathbb{R}\oplus t_{ji}'\mathbb{R},
\end{equation*}
respectively. Now, on the one-dimensional complex torus $T^2_{J=t_{ji}'}$, we consider the exact triangle
\begin{align*}
\begin{CD}
\cdots &@>>> E_{(1,0,(\tilde{\mu}_0')_i,\mathcal{U}'_1)} @>>> E_{(t',N,(\tilde{\eta}_0')_i,\mathcal{U}'_{t'})} @>>> E_{(s',N,(\tilde{\nu}_0')_i,\mathcal{U}'_{s'})} \\
&@>>> E_{(1,0,(\tilde{\mu}_0')_i,\mathcal{U}'_1)}[1] @>>> \cdots . 
\end{CD}
\end{align*}
Here, for the holomorphic vector bundle of the form $E_{(r',a',\mu',\mathcal{U})}\rightarrow T^2_{J=t_{ji}'}$ of rank $r'$ ($r'\in \mathbb{N}$, $a'\in \mathbb{Z}$, $gcd(r',a')=1$, $\mu' \in \mathbb{R}\oplus t_{ji}'\mathbb{R}$), we took
\begin{equation}
\mathcal{U}'_{r'}:=\left\{ V_j=V, \ U_i=U^{-a'} \in U(r') \right\} \label{U1}
\end{equation}
as $\mathcal{U}$. In particular,
\begin{equation*}
\mathcal{U}'_1=\Bigl\{ V_j=U_i=1 \in U(1) \Bigr\}. 
\end{equation*}
Then, \cite[Theorem 5.6]{D} states that the exact triangle (\ref{tr6}) is equivalent to the exact triangle
\begin{align*}
\begin{CD}
\cdots &@>>> \pi^* E_{(1,0,(\tilde{\mu}_0')_i,\mathcal{U}'_1)}\otimes E_{(1,O,\tilde{\mu},\mathcal{U}_1)} @>>> \pi^* E_{(t',N,(\tilde{\eta}_0')_i,\mathcal{U}'_{t'})}\otimes E_{(1,O,\tilde{\mu},\mathcal{U}_1)} \\
&@>>> \pi^* E_{(s',N,(\tilde{\nu}_0')_i,\mathcal{U}'_{s'})}\otimes E_{(1,O,\tilde{\mu},\mathcal{U}_1)} @>>> \pi^* E_{(1,0,(\tilde{\mu}_0')_i,\mathcal{U}'_1)}[1]\otimes E_{(1,O,\tilde{\mu},\mathcal{U}_1)} @>>> \cdots 
\end{CD}
\end{align*}
in $Tr(DG_{T^{2n}_{J=T'}})$ with a suitable holomorphic line bundle $E_{(1,O,\tilde{\mu},\mathcal{U}_1)}\in \mathrm{Pic}^0(T^{2n}_{J=T'})$.

\subsection{Preparations}
This subsection is devoted to the preparations of subsection 5.3. In particular, in subsection 5.3, we consider several autoequivalences on $Tr(DG_{T^{2n}_{J=T'}})$ in order to transform a given exact triangle to an exact triangle which is easy to treat, so we also explain such autoequivalences on $Tr(DG_{T^{2n}_{J=T'}})$ in this subsection.

Let us consider the exact triangle (\ref{tr2}) in $Tr(DG_{T^{2n}_{J=T}})$ with the assumption $C(\psi)\in DG_{T^{2n}_{J=T}}$, namely, we assume that there exists a holomorphic vector bundle $E_{(t',C',\eta',\mathcal{W})}$ such that $C(\psi)\cong E_{(t',C',\eta',\mathcal{W})}$ :
\begin{align}
\begin{CD}
\cdots &@>>> E_{(r',A',\mu',\mathcal{U})} @>>> E_{(t',C',\eta',\mathcal{W})} @>>> E_{(s',B',\nu',\mathcal{V})} \\
&@>>> E_{(r',A',\mu',\mathcal{U})}[1] @>>> \cdots . \label{tr7}
\end{CD}
\end{align}
Here, we do not assume $\mathrm{rank}\hspace{0.5mm}E_{(r',A',\mu',\mathcal{U})}=1$. Then, we have
\begin{equation*}
\mathrm{rank}\hspace{0.5mm}\alpha =1
\end{equation*}
by Theorem \ref{theo4.1}. Furthermore, Proposition \ref{d5.1} and Proposition \ref{d5.2} are generalized as follows (since we can prove the following two propositions similarly as in the cases of Proposition \ref{d5.1} and Proposition \ref{d5.2}, we omit the proofs of them).
\begin{proposition} \label{d5.1+}
Assume $\mathrm{rank}\hspace{0.5mm}\alpha=1$. Then, there exist two matrices $\mathcal{A}$, $\mathcal{D}\in SL(n;\mathbb{Z})$ such that 
\begin{equation*}
\mathcal{D}^t(r's'\alpha)\mathcal{A}=-NE_{ij},
\end{equation*}
where $N\in \mathbb{N}$.
\end{proposition}
\begin{proposition} \label{d5.2+}
We assume $\mathrm{rank}\hspace{0.5mm}\alpha=1$, and take a pair $(\mathcal{A}, \mathcal{D})$ of two matrices $\mathcal{A}$, $\mathcal{D}\in SL(n;\mathbb{Z})$ which satisfy the statement of Proposition \ref{d5.1+}. Then,
\begin{equation*}
(\mathcal{A}^{-1}T\mathcal{D})_{ji'}=0 \ (1\leq i' \not =i\leq n)
\end{equation*}
and
\begin{equation*}
\mathrm{Im}(\mathcal{A}^{-1}T\mathcal{D})_{ji}\not =0
\end{equation*}
hold.
\end{proposition}
We take the $n$-dimensional complex torus $T^{2n}_{J=T'}$ which is biholomorphic to $T^{2n}_{J=T}$, where $T':=\mathcal{A}^{-1}T\mathcal{D}$. Also in this case, the biholomorphic map $\varphi : T^{2n}_{J=T'}\stackrel{\sim}{\rightarrow} T^{2n}_{J=T}$ is given by the relation (\ref{phi}). By using the biholomorphic map $\varphi$, we obtain the pullback
\begin{align}
\begin{CD}
\cdots &@>>> E_{(r',\tilde{A}',\tilde{\mu}',\mathcal{U}')} @>>> E_{(t',\tilde{C}',\tilde{\eta}',\mathcal{W}')} @>>> E_{(s',\tilde{B}',\tilde{\nu}',\mathcal{V}')} \\
&@>>> E_{(r',\tilde{A}',\tilde{\mu}',\mathcal{U}')}[1] @>>> \cdots \label{tr8}
\end{CD}
\end{align}
of the exact triangle (\ref{tr7}), where the notations 
\begin{equation*}
\tilde{A}', \ \tilde{\mu}', \ \mathcal{U}', \ \tilde{B}', \ \tilde{\nu}', \ \mathcal{V}', \ \tilde{C}', \ \tilde{\eta}', \ \mathcal{W}'
\end{equation*}
are as in subsection 5.1. Thus, we may consider the exact triangle (\ref{tr8}) in $Tr(DG_{T^{2n}_{J=T'}})$ instead of the exact triangle (\ref{tr7}) in $Tr(DG_{T^{2n}_{J=T}})$. We will focus on the exact triangle (\ref{tr8}) in subsection 5.3.

Now, we explain several autoequivalences on $Tr(DG_{T^{2n}_{J=T'}})$ which will be used in subsection 5.3. Let $L\rightarrow T^{2n}_{J=T'}$ be a flat holomorphic line bundle, i.e., $L\in \mathrm{Pic}^0(T^{2n}_{J=T'})$. For each fixed $L\in \mathrm{Pic}^0(T^{2n}_{J=T'})$, we can associate the autoequivalence
\begin{equation*}
\Phi_L : Tr(DG_{T^{2n}_{J=T'}})\stackrel{\sim}{\rightarrow} Tr(DG_{T^{2n}_{J=T'}})
\end{equation*}
which is induced by the operator $\otimes L$. We denote the group of such autoequivalences $\Phi_L : Tr(DG_{T^{2n}_{J=T'}})\stackrel{\sim}{\rightarrow} Tr(DG_{T^{2n}_{J=T'}})$ by
\begin{equation*}
\mathcal{P}ic^0(T^{2n}_{J=T'}).
\end{equation*}

On the other hand, for a given autoequivalence $\check{\Psi} : Tr(Fuk_{\rm aff}(\check{T}^{2n}_{J=T'}))\stackrel{\sim}{\rightarrow} Tr(Fuk_{\rm aff}(\check{T}^{2n}_{J=T'}))$, if we assume that the homological mirror symmetry conjecture for $(T^{2n}_{J=T'}, \check{T}^{2n}_{J=T'})$ holds true, namely, assume that there exists an equivalence 
\begin{equation*}
F : Tr(Fuk_{\rm aff}(\check{T}^{2n}_{J=T'}))\stackrel{\sim}{\rightarrow} Tr(DG_{T^{2n}_{J=T'}})
\end{equation*}
as triangulated categories, there exists an autoequivalence $\Psi : Tr(DG_{T^{2n}_{J=T'}})\stackrel{\sim}{\rightarrow} Tr(DG_{T^{2n}_{J=T'}})$ uniquely such that the following diagram commutes :
\begin{equation*}
\begin{CD}
Tr(Fuk_{\rm aff}(\check{T}^{2n}_{J=T'})) @>F>> Tr(DG_{T^{2n}_{J=T'}}) \\
@V\check{\Psi}VV   @VV\Psi V  \\
Tr(Fuk_{\rm aff}(\check{T}^{2n}_{J=T'})) @>>F> Tr(DG_{T^{2n}_{J=T'}}).
\end{CD}
\end{equation*}
In particular, sometimes we can give a complex or algebraic geometric interpretation for such an autoequivalence $\Psi : Tr(DG_{T^{2n}_{J=T'}})\stackrel{\sim}{\rightarrow} Tr(DG_{T^{2n}_{J=T'}})$ (in general, to give a complex or algebraic interpretation in this context is a difficult problem). Here, as an example of an autoequivalence on $Tr(Fuk_{\rm aff}(\check{T}^{2n}_{J=T'}))$, we explain the autoequivalence on $Tr(Fuk_{\rm aff}(\check{T}^{2n}_{J=T'}))$ which is induced by the symplectic group action on $\check{T}^{2n}_{J=T'}$. We denote the local coordinates of $\check{T}^{2n}_{J=T'}$ by $(X^1,\cdots, X^n, Y^1,\cdots, Y^n)^t$, and set
\begin{equation*}
\check{X}:=(X^1,\cdots, X^n), \ \check{Y}:=(Y^1,\cdots, Y^n)^t.
\end{equation*}
Let us consider the symplectic group
\begin{align*}
Sp^{(-T'^{-1})^t}(2n;\mathbb{Z}):=&\biggl\{ \left( \begin{array}{ccc} g_{11} & g_{12} \\ g_{21} & g_{22} \end{array} \right)\in SL(2n;\mathbb{Z}) \ | \ g_{11}, g_{12}, g_{21}, g_{22}\in M(n;\mathbb{Z}), \\ &g_{11}^t(T'^{-1})^t g_{21}=(g_{11}^t(T'^{-1})^t g_{21})^t, \ g_{11}^t(T'^{-1})^t g_{22}=(g_{11}^t(T'^{-1})^t g_{22})^t, \\
& g_{21}^t T'^{-1}g_{12}-g_{11}^t (T'^{-1})^t g_{22}=(-T'^{-1})^t \biggr\}
\end{align*}
associated to $(-T'^{-1})^t$. For an element
\begin{equation*}
g:=\left( \begin{array}{ccc} g_{11} & g_{12} \\ g_{21} & g_{22} \end{array} \right)\in Sp^{(-T'^{-1})^t}(2n;\mathbb{Z}),
\end{equation*}
we define the $Sp^{(-T'^{-1})^t}(2n;\mathbb{Z})$ action on $\check{T}^{2n}_{J=T'}$ by
\begin{equation*}
\left( \begin{array}{ccc} \check{X} \\ \check{Y} \end{array} \right)\in \check{T}^{2n}_{J=T'} \longmapsto g\left( \begin{array}{ccc} \check{X} \\ \check{Y} \end{array} \right)\in \check{T}^{2n}_{J=T'}.
\end{equation*}
This $Sp^{(-T'^{-1})^t}(2n;\mathbb{Z})$ action defines a symplectic automorphism $\check{\psi}^g : \check{T}^{2n}_{J=T'}\stackrel{\sim}{\rightarrow} \check{T}^{2n}_{J=T'}$ which is given by
\begin{equation*}
\check{\psi}^g\left( \begin{array}{ccc} \check{X} \\ \check{Y} \end{array} \right)=g\left( \begin{array}{ccc} \check{X} \\ \check{Y} \end{array} \right).
\end{equation*}
Then, for an arbitrary object $(L_{(r,A,p)}, \mathcal{L}_{(r,A,p,q)})\in Fuk_{\rm aff}(\check{T}^{2n}_{J=T'})$, we can consider the object
\begin{equation*}
\left( (\check{\psi}^g)^{-1}(L_{(r,A,p)}), (\check{\psi}^g)^*\mathcal{L}_{(r,A,p,q)} \right)\in Fuk_{\rm aff}(\check{T}^{2n}_{J=T'}).
\end{equation*}
Therefore, the symplectic automorphism $\check{\psi}^g$ induces the autoequivalence on $Fuk_{\rm aff}(\check{T}^{2n}_{J=T'})$, and it leads the autoequivalence 
\begin{equation*}
\check{\Psi}^g : Tr(Fuk_{\rm aff}(\check{T}^{2n}_{J=T'}))\stackrel{\sim}{\rightarrow} Tr(Fuk_{\rm aff}(\check{T}^{2n}_{J=T'})).
\end{equation*}

Now, we define a subgroup $\widetilde{Sp}^{(-T'^{-1})^t}(2n;\mathbb{Z})$ of $Sp^{(-T'^{-1})^t}(2n;\mathbb{Z})$ as follows. Let us consider a pair $(\mathfrak{r}, \mathfrak{A})\in \mathbb{Z}\times M(n;\mathbb{Z})$ with the following conditions :
\begin{align}
&gcd(\mathfrak{r}, \mathrm{det}\mathfrak{A})=1, \ \mathrm{i}.\mathrm{e}., \ \mathrm{there} \ \mathrm{exist} \ k, l\in \mathbb{Z} \ \mathrm{such} \ \mathrm{that} \ k\mathfrak{r}+l\mathrm{det}\mathfrak{A}=1. \label{gcd} \\
&\mathrm{two} \ \mathrm{matrices} \ \mathfrak{A}, \ T \ \mathrm{satisfy} \ \mathrm{the} \ \mathrm{relation} \ \mathfrak{A}T=(\mathfrak{A}T)^t. \label{at} 
\end{align}
Then, we define
\begin{align*}
\widetilde{Sp}^{(-T'^{-1})^t}(2n;\mathbb{Z}):=&\biggl\{ \left( \begin{array}{ccc} k I_n & l\tilde{\mathfrak{A}} \\ -\mathfrak{A} & \mathfrak{r}I_n \end{array} \right) \in M(2n;\mathbb{Z}) \ | \\
&\mathfrak{r}\in \mathbb{Z} \ \mathrm{and} \ \mathfrak{A}\in M(n;\mathbb{Z}) \ \mathrm{satisfy} \ \mathrm{the} \ \mathrm{conditions} \ (\ref{gcd}), \ (\ref{at}). \biggr\},
\end{align*}
where $\tilde{\mathfrak{A}}$ denotes the cofactor matrix of $\mathfrak{A}$, namely, $\mathfrak{A}\tilde{\mathfrak{A}}=\tilde{\mathfrak{A}}\mathfrak{A}=\mathrm{det}\mathfrak{A}I_n$ holds. We can easily check that this $\widetilde{Sp}^{(-T'^{-1})^t}(2n;\mathbb{Z})$ is a subgroup of $Sp^{(-T'^{-1})^t}(2n;\mathbb{Z})$ by using the conditions (\ref{gcd}), (\ref{at}). Hence, for each matrix
\begin{equation*}
g(\mathfrak{r}, \mathfrak{A}):=\left( \begin{array}{ccc} k I_n & l\tilde{\mathfrak{A}} \\ -\mathfrak{A} & \mathfrak{r}I_n \end{array} \right) \in \widetilde{Sp}^{(-T'^{-1})^t}(2n;\mathbb{Z}),
\end{equation*}
we can associate the autoequivalence 
\begin{equation*}
\check{\Psi}^{g(\mathfrak{r}, \mathfrak{A})} : Tr(Fuk_{\rm aff}(\check{T}^{2n}_{J=T'}))\stackrel{\sim}{\rightarrow} Tr(Fuk_{\rm aff}(\check{T}^{2n}_{J=T'})). 
\end{equation*}
Also, under the assumption that the homological mirror symmetry conjecture for $(T^{2n}_{J=T'}, \check{T}^{2n}_{J=T'})$ holds true, we can consider the autoequivalence
\begin{equation*}
\Psi_{g(\mathfrak{r}, \mathfrak{A})} : Tr(DG_{T^{2n}_{J=T'}})\stackrel{\sim}{\rightarrow} Tr(DG_{T^{2n}_{J=T'}})
\end{equation*}
compatible with the triangulated functors $F : Tr(Fuk_{\rm aff}(\check{T}^{2n}_{J=T'}))\stackrel{\sim}{\rightarrow} Tr(DG_{T^{2n}_{J=T'}})$, $\check{\Psi}^{g(\mathfrak{r}, \mathfrak{A})} : Tr(Fuk_{\rm aff}(\check{T}^{2n}_{J=T'}))\stackrel{\sim}{\rightarrow} Tr(Fuk_{\rm aff}(\check{T}^{2n}_{J=T'}))$. In this context, we denote the group of autoequivalences $\Psi_{g(\mathfrak{r}, \mathfrak{A})} : Tr(DG_{T^{2n}_{J=T'}})\stackrel{\sim}{\rightarrow} Tr(DG_{T^{2n}_{J=T'}})$ by
\begin{equation*}
\mathrm{Aut}^{\widetilde{Sp}}(T^{2n}_{J=T'}).
\end{equation*}

Here, we explain why we focus on the subgroup $\widetilde{Sp}^{(-T'^{-1})^t}(2n;\mathbb{Z})$ instead of the symplectic group $Sp^{(-T'^{-1})^t}(2n;\mathbb{Z})$ itself. As explained in subsection 5.1, in \cite{D}, we transform the exact triangle (\ref{tr4}) by using the triangulated functor $Tr(DG_{T^{2n}_{J=T'}})\stackrel{\sim}{\rightarrow} Tr(DG_{T^{2n}_{J=T'}})$ which is induced by the operator $\otimes E_{(1,-\tilde{A}',0,\mathcal{U}_1)}$. This triangulated functor is interpreted as the autoequivalence $\Psi_{g(1,\tilde{A}')^{-1}}=\Psi_{g(1,-\tilde{A}')}\in \mathrm{Aut}^{\widetilde{Sp}}(T^{2n}_{J=T'})$ associated to the matrix
\begin{equation*}
g(1,\tilde{A}')=\left( \begin{array}{ccc} I_n & O \\ -\tilde{A}' & I_n \end{array} \right)\in \widetilde{Sp}^{(-T'^{-1})^t}(2n;\mathbb{Z}),
\end{equation*}
and the holomorphic line bundle $E_{(1,\tilde{A}',\tilde{\mu}',\mathcal{U}')}$ in the exact triangle (\ref{tr4}) is mapped to the flat holomorphic line bundle $E_{(1,O,\tilde{\mu}',\mathcal{U}')}$ by $\Psi_{g(1,-\tilde{A}')}$. In this sense, $\mathrm{Aut}^{\widetilde{Sp}}(T^{2n}_{J=T'})$ is a straightforward extension of the group of autoequivalences on $Tr(DG_{T^{2n}_{J=T'}})$ which is discussed in subsection 5.2 in \cite{D} to general settings. In fact, under the assumption that the homological mirror symmetry conjecture for $(T^{2n}_{J=T'}, \check{T}^{2n}_{J=T'})$ holds true, we can transform not only holomorphic line bundles but also holomorphic vector bundles of higher rank to flat holomorphic line bundles by considering $\mathrm{Aut}^{\widetilde{Sp}}(T^{2n}_{J=T'})$. For example, we set
\begin{equation*}
T'=\mathbf{i}\left( \begin{array}{ccc} 1 & 0 \\ 1 & 2 \end{array} \right), \ r=3, \ A=\left( \begin{array}{ccc} 1 & 1 \\ 0 & 2 \end{array} \right), \ p=q=0.
\end{equation*}
It is clear that $AT'=(AT')^t$ holds and $r'=9$. We can also take a suitable set $\mathcal{U}$ by \cite[Proposition 3.2]{kazushi3}. Then, we can obtain the autoequivalence $\Psi_{g(3,A)^{-1}} : Tr(DG_{T^4_{J=T'}})\stackrel{\sim}{\rightarrow} Tr(DG_{T^4_{J=T'}})$ associated to the matrix
\begin{equation*}
g(3,A)=\left( \begin{array}{ccc} I_2 & -\tilde{A} \\ -A & 3I_2 \end{array} \right)\in \widetilde{Sp}^{(-T'^{-1})^t}(4;\mathbb{Z})
\end{equation*}
via the homological mirror symmetry. By using this $\Psi_{g(3,A)^{-1}}$, we can check that the holomorphic vector bundle $E_{(3,A,0,\mathcal{U})}$ of rank $9$ is mapped to the flat holomorphic line bundle.

Now, in order to state our problem, we give the following definition.
\begin{definition}\label{deftr}
Let us consider an exact triangle
\begin{equation}
\begin{CD}
\cdots @>>> E_1 @>>> E_2 @>>> E_3 @>>> E_1[1] @>>> \cdots \label{deftrex}
\end{CD}
\end{equation}
consisting of $E_1$, $E_2$, $E_3 \in \mathrm{Ob}(DG_{T^{2n}_{J=T'}})$. Then, we say that the exact triangle $(\ref{deftrex})$ essentially comes from the one-dimensional complex torus $T^2_{J=\tau}$ if there exist a one-dimensional complex torus $T^2_{J=\tau}$ $(\tau \in \mathbb{H})$ and a holomorphic projection $\pi : T^{2n}_{J=T'}\rightarrow T^2_{J=\tau}$ such that the exact triangle $(\ref{deftrex})$ has the expression
\begin{equation*}
\begin{CD}
\cdots @>>> \Phi \left( \pi^* E_1' \right) @>>> \Phi \left( \pi^* E_2' \right) @>>> \Phi \left( \pi^* E_3' \right) @>>> \Phi \left( \pi^* E_1' \right)[1] @>>> \cdots 
\end{CD}
\end{equation*}
by $\Phi \in \langle \mathrm{Aut}^{\widetilde{Sp}}(T^{2n}_{J=T'}), \mathcal{P}ic^0(T^{2n}_{J=T'}) \rangle$ and an exact triangle
\begin{equation*}
\begin{CD}
\cdots @>>> E_1' @>>> E_2' @>>> E_3' @>>> E_1'[1] @>>> \cdots 
\end{CD}
\end{equation*}
consisting of $E_1'$, $E_2'$, $E_3' \in \mathrm{Ob}(DG_{T^2_{J=\tau}})$. Here, $\langle \mathrm{Aut}^{\widetilde{Sp}}(T^{2n}_{J=T'}), \mathcal{P}ic^0(T^{2n}_{J=T'}) \rangle$ denotes the smallest subgroup of the group $\mathrm{Aut}(Tr(DG_{T^{2n}_{J=T'}}))$ of autoequivalences on $Tr(DG_{T^{2n}_{J=T'}})$ containing the subset $\mathrm{Aut}^{\widetilde{Sp}}(T^{2n}_{J=T'}) \cup \mathcal{P}ic^0(T^{2n}_{J=T'}) \subseteq \mathrm{Aut}(Tr(DG_{T^{2n}_{J=T'}}))$.
\end{definition}
Concerning Definition \ref{deftr}, we will consider the following problem in subsection 5.3 (as mentioned in subsection 5.1, the following problem is already solved in the case $\mathrm{rank}\hspace{0.5mm}E_{(r',\tilde{A}',\tilde{\mu}',\mathcal{U}')}=1$ in \cite{D}). 
\begin{problem}\label{conj}
When does the exact triangle $(\ref{tr8})$ essentially come from a one-dimensional complex torus $?$
\end{problem}

\subsection{Main result}
The purpose of this subsection is to give an answer for Problem \ref{conj}. 

Our first goal is to prove Theorem \ref{mainth} which is a generalization of \cite[Theorem 5.6]{D} to the case that $\mathrm{rank}\hspace{0.5mm}E_{(r',\tilde{A}',\tilde{\mu}',\mathcal{U}')}$ is not necessarily 1. Before stating Theorem \ref{mainth}, we give a remark. As explained in subsection 5.2, the definition of $\mathrm{Aut}^{\widetilde{Sp}}(T^{2n}_{J=T'})$ depends on the homological mirror symmetry conjecture for $(T^{2n}_{J=T'}, \check{T}^{2n}_{J=T'})$. However, Theorem \ref{mainth} itself can be proved without the homological mirror symmetry. Let us denote the group of autoequivalences on $Tr(DG_{T^{2n}_{J=T'}})$ which is induced by the operator $\otimes L$ $(L\in \mathrm{Pic}(T^{2n}_{J=T'}))$ by
\begin{equation*}
\mathcal{P}ic(T^{2n}_{J=T'}).
\end{equation*}
We can regard this $\mathcal{P}ic(T^{2n}_{J=T'})$ as a subgroup of $\langle \mathrm{Aut}^{\widetilde{Sp}}(T^{2n}_{J=T'}), \mathcal{P}ic^0(T^{2n}_{J=T'}) \rangle$. In the proof of Theorem \ref{mainth}, we will actually use autoequivalences which are included in $\mathcal{P}ic(T^{2n}_{J=T'})$ only, so the discussions in the proof of Theorem \ref{mainth} are closed in the complex geometry side. 
\begin{theo}\label{mainth}
The exact triangle $(\ref{tr8})$ essentially comes from a one-dimensional complex torus if $r'=\mathrm{rank}\hspace{0.5mm}E_{(r',\tilde{A}',\tilde{\mu}',\mathcal{U}')}$ and $s'=\mathrm{rank}\hspace{0.5mm}E_{(s',\tilde{B}',\tilde{\nu}',\mathcal{V}')}$ are relatively prime, i.e., $gcd(r', s')=1$.
\end{theo}
\begin{proof}
By Proposition \ref{d5.1+}, we have
\begin{equation*}
\frac{1}{s'}\tilde{B}'=\frac{1}{r'}\tilde{A}'+\frac{N}{r's'}E_{ij},
\end{equation*}
and this fact leads the relation
\begin{equation}
\frac{\tilde{b}_{kl}'}{s'}=\frac{\tilde{a}_{kl}'}{r'} \label{Rel1}
\end{equation}
for $(k,l)\not=(i,j)$. Therefore, by the assumption $gcd(r',s')=1$, we see that there exist two integers $\tilde{a}_{kl}''$, $\tilde{b}_{kl}''\in \mathbb{Z}$ such that
\begin{equation*}
\tilde{a}_{kl}'=r'\tilde{a}_{kl}'', \ \tilde{b}_{kl}'=s'\tilde{b}_{kl}'',
\end{equation*}
and actually, each $\tilde{a}_{kl}''\in \mathbb{Z}$ coincides with $\tilde{b}_{kl}''\in \mathbb{Z}$ since the equality (\ref{Rel1}) holds. Hereafter, we denote
\begin{equation*}
a_{kl}'':=\tilde{a}_{kl}''=\tilde{b}_{kl}''\in \mathbb{Z}
\end{equation*}
for simplicity. By using these integers $a_{kl}''\in \mathbb{Z}$, let us define a matrix $A_{ij}''$ by
\begin{equation*}
A_{ij}'':=\left( \begin{array}{@{\,}cccccc@{\,}}
a_{11}'' & \ldots & a_{1j}'' & \ldots & a_{1n}'' & \\
\vdots & \ddots & \vdots & \ddots & \vdots \\
a_{i1}'' & \ldots & 0 & \ldots & a_{in}'' & \\
\vdots & \ddots & \vdots & \ddots & \vdots \\
a_{n1}'' & \ldots & a_{nj}'' & \ldots & a_{nn}''              
\end{array} \right)\in M(n;\mathbb{Z}).
\end{equation*}
Then, $\frac{1}{r'}\tilde{A}'$, $\frac{1}{s'}\tilde{B}'$ turns out to be
\begin{equation*}
\frac{1}{r'}\tilde{A}'=\frac{\tilde{a}_{ij}'}{r'}E_{ij}+A_{ij}'', \ \frac{1}{s'}\tilde{B}'=\frac{\tilde{b}_{ij}'}{s'}E_{ij}+A_{ij}'',
\end{equation*}
respectively, and
\begin{align*}
\frac{1}{t'}\tilde{C}'&=\frac{\tilde{c}_{ij}'}{t'}+\left( \frac{\tilde{c}_{kl}'}{t'} \right) \\
&=\frac{\tilde{c}_{ij}'}{t'}E_{ij}+\left( \frac{\tilde{a}_{kl}'}{t'}+\frac{\tilde{b}_{kl}'}{t'} \right) \\
&=\frac{\tilde{c}_{ij}'}{t'}E_{ij}+\left( \frac{r'+s'}{t'}a_{kl}'' \right) \\
&=\frac{\tilde{c}_{ij}'}{t'}E_{ij}+A_{ij}''.
\end{align*}
In particular, we can show that $\frac{\tilde{a}_{ij}'}{r'}$, $\frac{\tilde{b}_{ij}'}{s'}$, $\frac{\tilde{c}_{ij}'}{t'}\in \mathbb{Q}$ are irreducible fractions as follows. We focus on $\frac{\tilde{a}_{ij}'}{r'}\in \mathbb{Q}$. The equality
\begin{equation*}
\frac{1}{r'}\tilde{A}'=\frac{\tilde{a}_{ij}'}{r'}E_{ij}+A_{ij}''
\end{equation*}
implies
\begin{equation*}
E_{(r',\tilde{A}',\tilde{\mu}',\mathcal{U}')}\cong E_{(r',\tilde{a}_{ij}'E_{ij},\tilde{\mu}',\mathcal{U}')}\otimes E_{(1,A_{ij}'',0,\mathcal{U}_1)}.
\end{equation*}
Now, $E_{(r',\tilde{A}',\tilde{\mu}',\mathcal{U}')}$ is simple, and $\mathrm{rank}\hspace{0.5mm}E_{(1,A_{ij}'',0,\mathcal{U}_1)}=1$, so $E_{(r',\tilde{a}_{ij}'E_{ij},\tilde{\mu}',\mathcal{U}')}$ is also simple. Hence, by \cite[Proposition 5.5]{D}, we can conclude $gcd(r',\tilde{a}_{ij}')=1$. Similarly, we can show that $\frac{\tilde{b}_{ij}'}{s'}$, $\frac{\tilde{c}_{ij}'}{t'}\in \mathbb{Q}$ are irreducible fractions.

Here, by applying the autoequivalence 
\begin{equation*}
\Psi_{g(1,A_{ij}'')^{-1}}=\Psi_{g(1,-A_{ij}'')} \in \mathrm{Aut}^{\widetilde{Sp}}(T^{2n}_{J=T'})
\end{equation*}
associated to the matrix
\begin{equation*}
g(1,A_{ij}'')=\left( \begin{array}{ccc} I_n & O \\ -A_{ij}'' & I_n \end{array} \right)\in \widetilde{Sp}^{(-T'^{-1})^t}(2n;\mathbb{Z})
\end{equation*}
to the exact triangle (\ref{tr8}), the exact triangle (\ref{tr8}) turns out to be
\begin{align}
\begin{CD}
\cdots &@>>> E_{(r',\tilde{a}_{ij}'E_{ij},\tilde{\mu}',\mathcal{U}')} @>>> E_{(t',\tilde{c}_{ij}'E_{ij},\tilde{\eta}',\mathcal{W}')} @>>> E_{(s',\tilde{b}_{ij}'E_{ij},\tilde{\nu}',\mathcal{V}')} \\
&@>>> E_{(r',\tilde{a}_{ij}'E_{ij},\tilde{\mu}',\mathcal{U}')}[1] @>>> \cdots . \label{tr9}
\end{CD}
\end{align}
As mentioned in subsection 5.2, the above autoequivalence $\Psi_{g(1,-A_{ij}'')}$ is interpreted as the triangulated functor $Tr(DG_{T^{2n}_{J=T'}})\stackrel{\sim}{\rightarrow} Tr(DG_{T^{2n}_{J=T'}})$ which is induced by the operator $\otimes E_{(1,-A_{ij}'',0,\mathcal{U}_1)}$. Furthermore, we modify the exact triangle (\ref{tr9}) as follows. By \cite[Theorem 3.4]{kazushi3}, there exist
\begin{equation*}
\tilde{\mu}_0', \ \tilde{\nu}_0', \ \tilde{\eta}_0'\in \mathbb{R}^n\oplus T'^t \mathbb{R}^n
\end{equation*}
such that
\begin{align*}
&E_{(r',\tilde{a}_{ij}'E_{ij},\tilde{\mu}',\mathcal{U}')}\cong E_{(r',\tilde{a}_{ij}'E_{ij},\tilde{\mu}_0',\mathcal{U}_{r'})}, \ E_{(s',\tilde{b}_{ij}'E_{ij},\tilde{\nu}',\mathcal{V}')}\cong E_{(s',\tilde{b}_{ij}'E_{ij},\tilde{\nu}_0',\mathcal{U}_{s'})}, \\ 
&E_{(t',\tilde{c}_{ij}'E_{ij},\tilde{\eta}',\mathcal{W}')}\cong E_{(t',\tilde{c}_{ij}'E_{ij},\tilde{\eta}_0',\mathcal{U}_{t'})},
\end{align*}
so we may consider the following instead of the exact triangle (\ref{tr9}) :
\begin{align}
\begin{CD}
\cdots &@>>> E_{(r',\tilde{a}_{ij}'E_{ij},\tilde{\mu}_0',\mathcal{U}_{r'})} @>>> E_{(t',\tilde{c}_{ij}'E_{ij},\tilde{\eta}_0',\mathcal{U}_{t'})} @>>> E_{(s',\tilde{b}_{ij}'E_{ij},\tilde{\nu}_0',\mathcal{U}_{s'})} \\
&@>>> E_{(r',\tilde{a}_{ij}'E_{ij},\tilde{\mu}_0',\mathcal{U}_{r'})}[1] @>>> \cdots . \label{tr10}
\end{CD}
\end{align}
We prepare some notations for later convenience. We decompose $\tilde{\mu}_0'=\tilde{p}_0'+T'^t \tilde{q}_0'\in \mathbb{R}^n \oplus T'^t \mathbb{R}^n$, and define
\begin{align*}
(\tilde{\mu}_0')^{\vee i}:=&((\tilde{p}_0')_1,\cdots, (\tilde{p}_0)'_{i-1}, 0, (\tilde{p}_0')_{i+1},\cdots, (\tilde{p}_0')_n)^t \\
&+T'^t ((\tilde{q}_0')_1,\cdots, (\tilde{q}_0')_{j-1}, 0, (\tilde{q}_0)_{j+1},\cdots, (\tilde{q}_0')_n)^t \in \mathbb{R}^n \oplus T'^t \mathbb{R}^n, 
\end{align*}
\hspace{5mm} $(\tilde{\mu}_0')_i:=(\tilde{p}_0')_i+t_{ji}' (\tilde{q}_0')_j \in \mathbb{R}\oplus t_{ji}' \mathbb{R}.$ \\\\
Similarly, we use the notations $(\tilde{\nu}_0')^{\vee i}$, $(\tilde{\eta}_0')^{\vee i}\in \mathbb{R}^n\oplus T'^t \mathbb{R}^n$, $(\tilde{\nu}_0')_i$, $(\tilde{\eta}_0')_i\in \mathbb{R}\oplus t_{ji}'\mathbb{R}$ in the above sense. Here, we give a remark on the exact triangle (\ref{tr10}). In the exact triangle (\ref{tr10}), the non-triviality of
\begin{equation*}
\mathrm{Ext}^1(E_{(s',\tilde{b}_{ij}'E_{ij},\tilde{\nu}_0',\mathcal{U}_{s'})}, E_{(r',\tilde{a}_{ij}'E_{ij},\tilde{\mu}_0',\mathcal{U}_{r'})})
\end{equation*}
implies
\begin{align}
&\frac{1}{r'}\tilde{\mu}_0'^{\wedge i} \equiv \frac{1}{s'}\tilde{\nu}_0'^{\wedge i} \ (\mathrm{mod} \ 2\pi (\mathbb{Z}^{n-1} \oplus \tilde{T}_{ji}'^t \mathbb{Z}^{n-1})), \label{munu} \\
&\frac{1}{s'}\tilde{\nu}_0'^{\wedge i} \equiv \frac{1}{t'}\tilde{\eta}_0'^{\wedge i} \ (\mathrm{mod} \ 2\pi (\mathbb{Z}^{n-1} \oplus \tilde{T}_{ji}'^t \mathbb{Z}^{n-1})), \label{nueta}
\end{align}
where $\tilde{\mu}_0'^{\wedge i}$, $\tilde{\nu}_0'^{\wedge i}$, $\tilde{\eta}_0'^{\wedge i}\in \mathbb{R}^{n-1}\oplus \tilde{T}_{ji}'^t \mathbb{R}^{n-1}$ denote the constant vectors obtained by eliminating the $i$-th components from $\tilde{\mu}_0'$, $\tilde{\nu}_0'$, $\tilde{\eta}_0'$, respectively, and $\tilde{T}_{ji}'$ denotes the matrix obtained by eliminating the $j$-th row and the $i$-th column from $T'$ (see also p.32, p.33 in \cite{D}). In particular, without loss of generality we may assume
\begin{equation*}
\tilde{\mu}:=\frac{1}{r'}\tilde{\mu}_0'^{\vee i}=\frac{1}{s'}\tilde{\nu}_0'^{\vee i}=\frac{1}{t'}\tilde{\eta}_0'^{\vee i}\in \mathbb{R}^n \oplus T'^t \mathbb{R}^n.
\end{equation*}
Now, since the relations
\begin{equation*}
t_{ji'}'=0 \ (1\leq i'\not= i\leq n), \ \mathrm{Im}t_{ji}'\not=0
\end{equation*}
hold by Proposition \ref{d5.2+}, we can define the holomorphic projection $\pi : T^{2n}_{J=T'}\rightarrow \mathbb{C}/2\pi (\mathbb{Z}\oplus t_{ji}'\mathbb{Z})$ by
\begin{equation*}
\pi (Z)=Z_j=X_j+t_{ji}'Y_i.
\end{equation*}
Then, by the non-triviality of $\mathrm{Ext}^1(E_{(s',\tilde{b}_{ij}'E_{ij},\tilde{\nu}_0',\mathcal{U}_{s'})}, E_{(r',\tilde{a}_{ij}'E_{ij},\tilde{\mu}_0',\mathcal{U}_{r'})})$, we may assume $\mathrm{Im}t_{ji}'>0$, so hereafter, we denote $T^2_{J=t_{ji}'}=\mathbb{C}/2\pi (\mathbb{Z}\oplus t_{ji}'\mathbb{Z})$. Let us consider the exact triangle
\begin{align*}
\begin{CD}
\cdots &@>>> E_{(r',\tilde{a}_{ij}',(\tilde{\mu}_0')_i,\mathcal{U}'_{r'})} @>>> E_{(t',\tilde{c}_{ij}',(\tilde{\eta}_0')_i,\mathcal{U}'_{t'})} @>>> E_{(s',\tilde{b}_{ij}',(\tilde{\nu}_0')_i,\mathcal{U}'_{s'})} \\
&@>>> E_{(r',\tilde{a}_{ij}',(\tilde{\mu}_0')_i,\mathcal{U}'_{r'})}[1] @>>> \cdots  
\end{CD}
\end{align*}
in $Tr(DG_{T^2_{J=t_{ji}'}})$, where the notations $\mathcal{U}'_{r'}$, $\mathcal{U}'_{s'}$, $\mathcal{U}'_{t'}$ are used in the sense of the set (\ref{U1}). By using the holomorphic projection $\pi : T^{2n}_{J=T'}\rightarrow T^2_{J=t_{ji}'}$, the following exact triangle in $Tr(DG_{T^{2n}_{J=T'}})$ is induced from the above exact triangle :
\begin{align}
\begin{CD}
\cdots &@>>> \pi^* E_{(r',\tilde{a}_{ij}',(\tilde{\mu}_0')_i,\mathcal{U}'_{r'})} @>>> \pi^* E_{(t',\tilde{c}_{ij}',(\tilde{\eta}_0')_i,\mathcal{U}'_{t'})} @>>> \pi^* E_{(s',\tilde{b}_{ij}',(\tilde{\nu}_0')_i,\mathcal{U}'_{s'})} \\
&@>>> \pi^* E_{(r',\tilde{a}_{ij}',(\tilde{\mu}_0')_i,\mathcal{U}'_{r'})}[1] @>>> \cdots . \label{tr11} 
\end{CD}
\end{align}
We take the autoequivalence
\begin{equation*}
\Phi_{E_{(1,O,\tilde{\mu}, \mathcal{U}_1)}}\in \mathcal{P}ic^0(T^{2n}_{J=T'})
\end{equation*}
associated to the operator $\otimes E_{(1,O,\tilde{\mu}, \mathcal{U}_1)}$. The exact triangle which is obtained by applying $\Phi_{E_{(1,O,\tilde{\mu}, \mathcal{U}_1)}} : Tr(DG_{T^{2n}_{J=T'}})\stackrel{\sim}{\rightarrow} Tr(DG_{T^{2n}_{J=T'}})$ to the exact triangle (\ref{tr11}) is indeed the exact triangle (\ref{tr10}) itself. This completes the proof.
\end{proof}
Although Theorem \ref{mainth} does not depend on the homological mirror symmetry conjecture for $(T^{2n}_{J=T'}, \check{T}^{2n}_{J=T'})$, we can discuss Problem \ref{conj} also in the case $gcd(r',s')\not=1$ via the homological mirror symmetry. However, unfortunately, there exists an exact triangle (\ref{tr8}) that essentially does not come from a one-dimensional complex torus in such cases. In order to construct a counterexample for Problem \ref{conj} in the case $gcd(r',s')\not=1$, we first prove the following lemma.
\begin{lemma}\label{lem}
For an arbitrary irreducible fraction $\mathfrak{r}'\in \mathbb{Q}$ and the matrix
\begin{equation*}
A:=\left( \begin{array}{ccc} 1 & 1 \\ 0 & 2 \end{array} \right),
\end{equation*}
there does not exist a pair $(\mathfrak{r}, \mathfrak{A})\in \mathbb{Z}\times M(2;\mathbb{Z})$ which satisfies the condition $(\ref{gcd})$ and the relation
\begin{equation*}
-2\mathfrak{A}+\mathfrak{r}A=\mathfrak{r}'E_{11}(2kI_2+l\tilde{\mathfrak{A}}A).
\end{equation*}
\end{lemma}
\begin{proof}
First, let us consider the case $\mathfrak{r}'=0$. We assume that there exists a pair $(\mathfrak{r}, \mathfrak{A})\in \mathbb{Z}\times M(2;\mathbb{Z})$ which satisfies the condition (\ref{gcd}) and the relation
\begin{equation}
-2\mathfrak{A}+\mathfrak{r}A=O. \label{rel1}
\end{equation}
It is clear that the relation (\ref{rel1}) turns out to be
\begin{equation}
\mathfrak{A}=\frac{\mathfrak{r}}{2}\left( \begin{array}{ccc} 1 & 1 \\ 0 & 2 \end{array}    \right), \label{rel2}
\end{equation}
so the condition $\mathfrak{A}\in M(2;\mathbb{Z})$ indicates the existence of an integer $\mathfrak{r}''\in \mathbb{Z}$ such that
\begin{equation*}
\mathfrak{r}=2\mathfrak{r}''\in 2\mathbb{Z}.
\end{equation*}
Now, we substitute the relation $\mathfrak{r}=2\mathfrak{r}''$ to the equality (\ref{rel2}). As a result, although we have
\begin{equation*}
\mathrm{det}\mathfrak{A}=\mathfrak{r}\mathfrak{r}'',
\end{equation*}
it implies
\begin{equation*}
gcd(\mathfrak{r}, \mathrm{det}\mathfrak{A})=\mathfrak{r}=1.
\end{equation*}
This fact contradicts the condition $\mathfrak{r}\in 2\mathbb{Z}$.

Let us consider the case $\mathfrak{r}'\not=0$. We assume that there exists a pair $(\mathfrak{r}, \mathfrak{A})\in \mathbb{Z}\times M(2;\mathbb{Z})$ which satisfies the condition (\ref{gcd}) and the relation
\begin{equation}
-2\mathfrak{A}+\mathfrak{r}A=\mathfrak{r}'E_{11}(2kI_2+l\tilde{\mathfrak{A}}A). \label{rel3}
\end{equation}
For simplicity, we set
\begin{equation*}
\mathscr{A}=(\mathscr{A}_{ij}):=2kI_2+l\tilde{\mathfrak{A}}A.
\end{equation*}
Then, 
\begin{equation*}
\mathfrak{r}'E_{11}\mathscr{A}=\left( \begin{array}{ccc} \mathfrak{r}'\mathscr{A}_{11} & \mathfrak{r}'\mathscr{A}_{12} \\ 0 & 0 \end{array} \right), \ -2\mathfrak{A}+\mathfrak{r}A=\left( \begin{array}{ccc} \mathfrak{r}-2\mathfrak{A}_{11}  & \mathfrak{r}-2\mathfrak{A}_{12} \\ -2\mathfrak{A}_{21} & 2\mathfrak{r}-2\mathfrak{A}_{22} \end{array} \right),
\end{equation*}
so we obtain 
\begin{equation*}
\mathfrak{A}_{21}=0, \ \mathfrak{A}_{22}=\mathfrak{r}.
\end{equation*}
This fact indicates
\begin{equation}
\mathrm{det}\mathfrak{A}=\mathrm{det}\left( \begin{array}{ccc} \mathfrak{A}_{11} & \mathfrak{A}_{12} \\ 0 & \mathfrak{r} \end{array} \right)=\mathfrak{r}\mathfrak{A}_{11}. \label{rel4}
\end{equation}
By using the equality (\ref{rel4}), we can rewrite the condition (\ref{gcd}) to
\begin{equation*}
gcd(\mathfrak{r}, \mathrm{det}\mathfrak{A})=gcd(\mathfrak{r}, \mathfrak{r}\mathfrak{A}_{11})=\mathfrak{r}=1,
\end{equation*}
namely, we have
\begin{equation*}
\mathfrak{A}_{22}=\mathfrak{r}=1.
\end{equation*}
Therefore, we see
\begin{align*}
\mathscr{A}&=\left( \begin{array}{ccc} 2k & 0 \\ 0 & 2k \end{array} \right)+l \left( \begin{array}{ccc} 1 & -\mathfrak{A}_{12} \\ 0 & \mathfrak{A}_{11} \end{array} \right) \left( \begin{array}{ccc} 1 & 1 \\ 0 & 2 \end{array} \right) \\
&=\left( \begin{array}{ccc} 2k+l & l(1-2\mathfrak{A}_{12}) \\ 0 & 2k+2l\mathfrak{A}_{11} \end{array} \right),
\end{align*}
and the equality (\ref{rel3}) turns out to be
\begin{equation}
\left( \begin{array}{ccc} 1-2\mathfrak{A}_{11} & 1-2\mathfrak{A}_{12} \\ 0 & 0 \end{array} \right)=\left( \begin{array}{ccc} \mathfrak{r}'(2k+l) & \mathfrak{r}'l(1-2\mathfrak{A}_{12}) \\ 0 & 0 \end{array} \right). \label{rel5}
\end{equation}
In particular, by focusing on the (1,2) component of the equality (\ref{rel5}), we obtain
\begin{equation}
2\mathfrak{A}_{12}(1-\mathfrak{r}'l)=1-\mathfrak{r}'l. \label{rel6}
\end{equation}
Suppose $1-\mathfrak{r}'l\not= 0$. Then, the equality (\ref{rel6}) indicates
\begin{equation*}
\mathfrak{A}_{12}=\frac{1}{2}\not \in \mathbb{Z},
\end{equation*}
and this fact contradicts the condition $\mathfrak{A}\in M(2;\mathbb{Z})$. Suppose $1-\mathfrak{r}'l=0$, namely, $\mathfrak{r}'l=1$. By focusing on the (1,1) component of the equality (\ref{rel5}) under the assumption $\mathfrak{r}'l=1$, we see
\begin{equation*}
\mathfrak{A}_{11}=-\mathfrak{r}'k.
\end{equation*}
Then, unfortunately, the left hand side of the relation $k\mathfrak{r}+l\mathrm{det}\mathfrak{A}=1$ in the condition (\ref{gcd}) turns out to be
\begin{equation*}
k\mathfrak{r}+l\mathrm{det}\mathfrak{A}=k+l\mathrm{det}\mathfrak{A}=k+l (-\mathfrak{r}'k)=k-k=0,
\end{equation*}
so the relation $\mathfrak{r}'l=1$ does not compatible with the condition (\ref{gcd}). This completes the proof.
\end{proof}
We construct a counterexample for Problem \ref{conj} in the case $gcd(r', s')\not=1$. We set
\begin{align*}
&T':=\mathbf{i} \left( \begin{array}{ccc} 1 & 0 \\ 1 & 2 \end{array} \right), \ r:=2, \ A:=\left( \begin{array}{ccc} 1 & 1 \\ 0 & 2 \end{array} \right), \ s:=2, \ B:=\left( \begin{array}{ccc} 2 & 1 \\ 0 & 2 \end{array} \right), \\
&t:=4, \ C:=\left( \begin{array}{ccc} 3 & 2 \\ 0 & 4 \end{array} \right).
\end{align*}
Then, it is clear that $AT'=(AT')^t$, $BT'=(BT')^t$, $CT'=(CT')^t$ hold and $r'=2$, $s'=2$, $t'=4$. Moreover, we can also verify the following :
\begin{equation*}
r'+s'=t', \ \frac{r'}{r}A+\frac{s'}{s}B=\frac{t'}{t}C, \ \mathrm{rank}\hspace{0.5mm}\alpha=1.
\end{equation*}
Let us consider the triangle
\begin{align}
\begin{CD}
\cdots &@>>> E_{(2,A,\mu,\mathcal{U})} @>>> E_{(4,C,\eta,\mathcal{W})} @>>> E_{(2,B,\nu,\mathcal{V})} \\
&@>>> E_{(2,A,\mu,\mathcal{U})}[1] @>>> \cdots , \label{TR1}
\end{CD}
\end{align}
where $\mu$, $\nu$, $\eta \in \mathbb{R}^2\oplus T'^t \mathbb{R}^2$ and $\mathcal{U}$, $\mathcal{V}$, $\mathcal{W}$ denote the sets in the sense of the definition (\ref{setU}). 

We can check that there exist suitable parameters $\mu'$, $\nu'$, $\eta' \in \mathbb{R}^2 \oplus T'^t \mathbb{R}^2$ and suitable sets $\mathcal{U}'$, $\mathcal{V}'$, $\mathcal{W}'$ such that the triangle (\ref{TR1}) becomes an exact triangle
\begin{align*}
\begin{CD}
\cdots &@>>> E_{(2,A,\mu',\mathcal{U}')} @>>> E_{(4,C,\eta',\mathcal{W}')} @>>> E_{(2,B,\nu',\mathcal{V}')} \\
&@>>> E_{(2,A,\mu',\mathcal{U}')}[1] @>>> \cdots 
\end{CD}
\end{align*}
as follows. We first consider the dual complex torus
\begin{equation*}
\hat{T}^4_{J=T'}=\mathbb{C}^2/2\pi (\mathbb{Z}^2\oplus T'^t \mathbb{Z}^2)
\end{equation*}
of the complex torus $T^4_{J=T'}$ and its mirror dual $(\hat{T}^4_{J=T'})^{\vee}$. Let us denote the local coordinates of $(\hat{T}^4_{J=T'})^{\vee}$ by 
\begin{equation*}
\left( \begin{array}{ccc} \check{X} \\ \check{Y} \end{array} \right),
\end{equation*}
where $\check{X}:=(X^1, X^2)^t$, $\check{Y}:=(Y^1, Y^2)^t$. Then, the complexified symplectic form of $(\hat{T}^4_{J=T'})^{\vee}$ is expressed locally as
\begin{equation*}
d\check{X}^t (-T'^{-1}) d\check{Y}.
\end{equation*}
We set
\begin{equation*}
g:=\left( \begin{array}{ccc} O & I_2 \\ -I_2 & O \end{array} \right)\in SL(4;\mathbb{Z}),
\end{equation*}
and define a symplectic morphism $\varphi^g : (\hat{T}^4_{J=T'})^{\vee} \stackrel{\sim}{\rightarrow} \check{T}^4_{J=T'}$ by
\begin{equation*}
\varphi^g \left( \begin{array}{ccc} \check{X} \\ \check{Y} \end{array} \right)
=g \left( \begin{array}{ccc} \check{X} \\ \check{Y} \end{array} \right).
\end{equation*}
Similarly as in the discussions in subsection 5.2, this symplectic morphism $\varphi^g$ induces the equivalence
\begin{equation*}
\Phi^g : Tr(Fuk_{\rm aff}(\check{T}^4_{J=T'}))\stackrel{\sim}{\rightarrow} Tr(Fuk_{\rm aff}((\hat{T}^4_{J=T'})^{\vee}))
\end{equation*}
as triangulated categories. This triangulated functor $\Phi^g$ corresponds to the Fourier-Mukai transform 
\begin{equation*}
\Phi_{\mathcal{P}} : Tr(DG_{T^4_{J=T'}})\stackrel{\sim}{\rightarrow} Tr(DG_{\hat{T}^4_{J=T'}})
\end{equation*}
associated to the Poincar\'{e} line bundle $\mathcal{P}\rightarrow T^4_{J=T'}\times \hat{T}^4_{J=T'}$ via the homological mirror symmetry. Namely, for equivalences
\begin{equation*}
F : Tr(DG_{T^4_{J=T'}})\stackrel{\sim}{\rightarrow} Tr(Fuk_{\rm aff}(\check{T}^4_{J=T'})), \ \hat{F} : Tr(DG_{\hat{T}^4_{J=T'}})\stackrel{\sim}{\rightarrow} Tr(Fuk_{\rm aff}((\hat{T}^4_{J=T'})^{\vee}))
\end{equation*}
as triangulated categories, the following diagram commutes :
\begin{equation*}
\begin{CD}
Tr(DG_{T^4_{J=T'}}) @>F>> Tr(Fuk_{\rm aff}(\check{T}^4_{J=T'})) \\
@V\Phi_{\mathcal{P}}VV   @VV\Phi^g V \\
Tr(DG_{\hat{T}^4_{J=T'}}) @>>\hat{F}> Tr(Fuk_{\rm aff}((\hat{T}^4_{J=T'})^{\vee})).
\end{CD}
\end{equation*}
Note that $\Phi_{\mathcal{P}}\not \in \langle \mathrm{Aut}^{\widetilde{Sp}}(T^4_{J=T'}), \mathcal{P}ic^0(T^4_{J=T'}) \rangle$, namely, the Fourier-Mukai transform $\Phi_{\mathcal{P}}$ is not an autoequivalence on $Tr(DG_{T^4_{J=T'}})$. Hence, by regarding the Fourier-Mukai transform $\Phi_{\mathcal{P}}$ as the triangulated functor $\hat{F}^{-1}\circ \Phi^g \circ F$, we can rewrite the triangle
\begin{align*}
\begin{CD}
\cdots &@>>> \Phi_{\mathcal{P}}(E_{(2,A,\mu,\mathcal{U})}) @>>> \Phi_{\mathcal{P}}(E_{(4,C,\eta,\mathcal{W})}) @>>> \Phi_{\mathcal{P}}(E_{(2,B,\nu,\mathcal{V})}) \\
&@>>> \Phi_{\mathcal{P}}(E_{(2,A,\mu,\mathcal{U})})[1] @>>> \cdots 
\end{CD}
\end{align*}
in $Tr(DG_{\hat{T}^4_{J=T'}})$ to the triangle
\begin{align}
\begin{CD}
\cdots &@>>> E_{(1,\tilde{A},\tilde{\mu},\tilde{\mathcal{U}})} @>>> E_{(3,\tilde{C},\tilde{\eta},\tilde{\mathcal{W}})} @>>> E_{(2,\tilde{B},\tilde{\nu},\tilde{\mathcal{V}})} \\
&@>>> E_{(1,\tilde{A},\tilde{\mu},\tilde{\mathcal{U}})}[1] @>>> \cdots , \label{TR2}
\end{CD}
\end{align}
where
\begin{equation*}
\tilde{A}:=\left( \begin{array}{ccc} -2 & 1 \\ 0 & -1 \end{array} \right), \ \tilde{B}:=\left( \begin{array}{ccc} -2 & 1 \\ 0 & -2 \end{array} \right), \ \tilde{C}:=\left( \begin{array}{ccc} -4 & 2 \\ 0 & -3 \end{array} \right), \ \tilde{\mu}, \tilde{\nu}, \tilde{\eta}\in \mathbb{R}^2\oplus T'\mathbb{R}^2
\end{equation*}
and $\tilde{\mathcal{U}}$, $\tilde{\mathcal{V}}$, $\tilde{\mathcal{W}}$ denote the sets in the sense of the definition (\ref{setU}). In particular, $\mathrm{rank}\hspace{0.5mm}E_{(1,\tilde{A},\tilde{\mu},\tilde{\mathcal{U}})}=1$, $\mathrm{rank}\hspace{0.5mm}E_{(2,\tilde{B},\tilde{\nu},\tilde{\mathcal{V}})}=2$, $\mathrm{rank}\hspace{0.5mm}E_{(3,\tilde{C},\tilde{\eta},\tilde{\mathcal{W}})}=3$. Then, since 
\begin{equation*}
\tilde{A}-\frac{1}{2}\tilde{B}=\left( \begin{array}{ccc} -1 & \frac{1}{2} \\ 0 & 0 \end{array} \right)
\end{equation*}
holds, we can use Proposition \ref{d5.1+} and Proposition \ref{d5.2+}. Actually, two matrices $\mathcal{A}$, $\mathcal{D}\in SL(2;\mathbb{Z})$ in Proposition \ref{d5.1+} are given by
\begin{equation*}
\mathcal{A}:=\left( \begin{array}{ccc} 1 & 1 \\ 1 & 2 \end{array} \right), \ \mathcal{D}:=I_2,
\end{equation*}
and the deformation of $T'$ in Proposition \ref{d5.2+} is described as
\begin{equation*}
\mathcal{A}^{-1}T'\mathcal{D}=\mathbf{i} \left( \begin{array}{ccc} 2 & 0 \\ -1 & 1 \end{array} \right).
\end{equation*}
Furthermore, by using these matrices $\mathcal{A}$, $\mathcal{D}\in SL(2;\mathbb{Z})$, we can transform three matrices $\tilde{A}$, $\tilde{B}$, $\tilde{C}$ to
\begin{equation*}
\mathcal{D}^t\tilde{A}\mathcal{A}=\left( \begin{array}{ccc} -1 & 0 \\ -1 & -2 \end{array} \right), \ \mathcal{D}^t\tilde{B}\mathcal{A}=\left( \begin{array}{ccc} -1 & 0 \\ -2 & -4 \end{array} \right), \ \mathcal{D}^t\tilde{C}\mathcal{A}=\left( \begin{array}{ccc} -2 & 0 \\ -3 & -6 \end{array} \right),
\end{equation*}
respectively. These facts imply that the triangle (\ref{TR2}) is induced from the triangle
\begin{align}
\begin{CD}
\cdots &@>>> E_{(1,-1,\tilde{\mu}',\mathcal{U}_1')} @>>> E_{(3,-2,\tilde{\eta}',\mathcal{U}_3')} @>>> E_{(2,-1,\tilde{\nu}',\mathcal{U}_2')} \\
&@>>> E_{(1,-1,\tilde{\mu}',\mathcal{U}_1')}[1] @>>> \cdots \label{TR3}
\end{CD}
\end{align}
on the one-dimensional complex torus $T^2_{J=2\mathbf{i}}=\mathbb{C}/2\pi(\mathbb{Z}\oplus 2\mathbf{i}\mathbb{Z})$, where $\tilde{\mu}'$, $\tilde{\nu}'$, $\tilde{\eta}'\in \mathbb{R}\oplus 2\mathbf{i}\mathbb{R}$ and each $\mathcal{U}_k'$ ($k=1,2,3$) denotes the set which is given in the definition (\ref{U1}) with $i=j=1$. In particular, since 
\begin{equation*}
\mathrm{dim}\mathrm{Ext}^1(E_{(2,-1,\tilde{\nu}',\mathcal{U}_2')}, E_{(1,-1,\tilde{\mu}',\mathcal{U}_1')})=1
\end{equation*}
holds by \cite[Proposition 3.3]{kazushi}, we see that there exist suitable parameters
\begin{equation*}
\tilde{\mu}_0', \ \tilde{\nu}_0', \ \tilde{\eta}_0'\in \mathbb{R}\oplus 2\mathbf{i}\mathbb{R}
\end{equation*}
such that the triangle (\ref{TR3}) becomes the exact triangle
\begin{align*}
\begin{CD}
\cdots &@>>> E_{(1,-1,\tilde{\mu}_0',\mathcal{U}_1')} @>>> E_{(3,-2,\tilde{\eta}_0',\mathcal{U}_3')} @>>> E_{(2,-1,\tilde{\nu}_0',\mathcal{U}_2')} \\
&@>>> E_{(1,-1,\tilde{\mu}_0',\mathcal{U}_1')}[1] @>>> \cdots 
\end{CD}
\end{align*}
by \cite[Theorem 4.10]{kazushi} (see also section 6 in \cite{kazushi}). Thus, we can conclude that there exist suitable parameters $\mu'$, $\nu'$, $\eta'\in \mathbb{R}^2\oplus T'^t \mathbb{R}^2$ and suitable sets $\mathcal{U}'$, $\mathcal{V}'$, $\mathcal{W}'$ such that the triangle (\ref{TR1}) becomes the exact triangle
\begin{align}
\begin{CD}
\cdots &@>>> E_{(2,A,\mu',\mathcal{U}')} @>>> E_{(4,C,\eta',\mathcal{W}')} @>>> E_{(2,B,\nu',\mathcal{V}')} \\
&@>>> E_{(2,A,\mu',\mathcal{U}')}[1] @>>> \cdots \label{TR4}
\end{CD}
\end{align}
under the assumption that the homological mirror symmetry conjecture for $(T^4_{J=T'}, \check{T}^4_{J=T'})$ holds true. In these discussions, indeed, we can also regard the (exact) triangle (\ref{TR1}) as the exact triangle which is induced from the (exact) triangle (\ref{TR3}) on the one-dimensional complex torus $T^2_{J=2\mathbf{i}}$. However, as mentioned above, \textbf{we use the Fourier-Mukai transform}
\begin{equation*}
\Phi_{\mathcal{P}}\not \in \langle \mathrm{Aut}^{\widetilde{Sp}}(T^4_{J=T'}), \mathcal{P}ic^0(T^4_{J=T'}) \rangle
\end{equation*}
\textbf{which is not included in the group of autoequivalences on $Tr(DG_{T^4_{J=T'}})$ when we transform the triangle (\ref{TR1}) to the triangle (\ref{TR2})}. Thus, in these discussions, we can not conclude that the (exact) triangle (\ref{TR1}) essentially comes from (the (exact) triangle (\ref{TR3}) on) the one-dimensional complex torus $T^2_{J=2\mathbf{i}}$ in the sense of Definition \ref{deftr}.

The remaining thing to be checked is that the exact triangle (\ref{TR4}) gives a counterexample for Problem \ref{conj}. We take an arbitrary autoequivalence $\Psi_{g(\mathfrak{r},\mathfrak{A})^{-1}}\in \mathrm{Aut}^{\widetilde{Sp}}(T^4_{J=T'})$ associated to 
\begin{equation*}
g(\mathfrak{r},\mathfrak{A})=\left( \begin{array}{ccc} kI_2 & l\tilde{\mathfrak{A}} \\ -\mathfrak{A} & \mathfrak{r}I_2 \end{array} \right)\in \widetilde{Sp}^{(-T'^{-1})^t}(4;\mathbb{Z}).
\end{equation*}
Then, it is enough to check that
\begin{equation*}
\Psi_{g(\mathfrak{r},\mathfrak{A})^{-1}}(E_{(2,A,\mu',\mathcal{U}')})
\end{equation*}
does not have the expression of the form
\begin{equation*}
E_{(n,aE_{11},\mu'',\mathcal{U}'')}.
\end{equation*}
Here, $n\in \mathbb{N}$ and $a\in \mathbb{Z}$ are relatively prime, i.e., $gcd(n,a)=1$, $\mu''\in \mathbb{R}^2\oplus T'^t \mathbb{R}^2$, and $\mathcal{U}''$ is a set which is defined by using the data $(n, aE_{11})\in \mathbb{N}\times M(2;\mathbb{Z})$. Now, note that the autoequivalence 
\begin{equation*}
\check{\Psi}^{g(\mathfrak{r},\mathfrak{A})^{-1}} : Tr(Fuk_{\rm aff}(\check{T}^4_{J=T'}))\stackrel{\sim}{\rightarrow} Tr(Fuk_{\rm aff}(\check{T}^4_{J=T'}))
\end{equation*}
compatible with the triangulated functors 
\begin{equation*}
F : Tr(DG_{T^4_{J=T'}})\stackrel{\sim}{\rightarrow} Tr(Fuk_{\rm aff}(\check{T}^4_{J=T'})), \ \Psi_{g(\mathfrak{r},\mathfrak{A})^{-1}} : Tr(DG_{T^4_{J=T'}})\stackrel{\sim}{\rightarrow} Tr(DG_{T^4_{J=T'}}), 
\end{equation*}
i.e., $F\circ \Psi_{g(\mathfrak{r},\mathfrak{A})^{-1}}=\check{\Psi}^{g(\mathfrak{r},\mathfrak{A})^{-1}}\circ F$ holds. We can express the object $F(E_{(2,A,\mu',\mathcal{U}')})\in \mathrm{Ob}(Fuk_{\rm aff}(\check{T}^4_{J=T'}))$ as
\begin{equation*}
(L_{(2,A,\check{p}')}, \mathcal{L}_{(2,A,\check{p}',\check{q}')})
\end{equation*}
by using the suitable parameters $\check{p}'$, $\check{q}'\in \mathbb{R}^2$. Suppose that
\begin{equation*}
\check{\Psi}^{g(\mathfrak{r},\mathfrak{A})^{-1}}(F(E_{(2,A,\mu',\mathcal{U}')}))
\end{equation*}
has the expression of the form
\begin{equation*}
(L_{(n,aE_{11},\check{p}'')}, \mathcal{L}_{(n,aE_{11},\check{p}'',\check{q}'')}),
\end{equation*}
where the notations $n\in \mathbb{N}$, $a\in \mathbb{Z}$ are as in the above, and $\check{p}''$, $\check{q}''\in \mathbb{R}^2$. Let us consider the transformation
\begin{align*}
g(\mathfrak{r},\mathfrak{A})\left( \begin{array}{ccc} \check{x} \\ \frac{1}{2}A\check{x}+\frac{1}{2}\check{p}' \end{array} \right)&=\left( \begin{array}{ccc} kI_2 & l\tilde{\mathfrak{A}} \\ -\mathfrak{A} & \mathfrak{r}I_2 \end{array} \right) \left( \begin{array}{ccc} \check{x} \\ \frac{1}{2}A\check{x}+\frac{1}{2}\check{p}' \end{array} \right) \\
&=\left( \begin{array}{ccc} \left( kI_2+\frac{l}{2}\tilde{\mathfrak{A}}A \right)\check{x}+\frac{l}{2}\tilde{\mathfrak{A}}\check{p}' \\ \left( -\mathfrak{A}+\frac{\mathfrak{r}}{2}A \right)\check{x}+\frac{\mathfrak{r}}{2}\check{p}' \end{array} \right)
\end{align*}
of the Lagrangian submanifold $L_{(2,A,\check{p}')}$ by $\check{\Psi}^{g(\mathfrak{r},\mathfrak{A})^{-1}}$. We set
\begin{align*}
&\check{X}=\left( kI_2+\frac{l}{2}\tilde{\mathfrak{A}}A \right)\check{x}+\frac{l}{2}\tilde{\mathfrak{A}}\check{p}'=\frac{1}{2}(2kI_2+l\tilde{\mathfrak{A}}A)\check{x}+\frac{l}{2}\tilde{\mathfrak{A}}\check{p}', \\
&\check{Y}=\left( -\mathfrak{A}+\frac{\mathfrak{r}}{2}A \right)\check{x}+\frac{\mathfrak{r}}{2}\check{p}'=\frac{1}{2}(-2\mathfrak{A}+\mathfrak{r}A)\check{x}+\mathfrak{r}{2}\check{p}'.
\end{align*}
Since $\check{\Psi}^{g(\mathfrak{r},\mathfrak{A})^{-1}}(L_{(2,A,\check{p}')})$ has the expression of the form $L_{(n,aE_{11},\check{p}'')}$ by the assumption, we may assume that the matrix $2kI_2+l\tilde{\mathfrak{A}}A$ has the inverse $(2kI_2+l\tilde{\mathfrak{A}}A)^{-1}$, and the relation
\begin{equation}
\check{Y}=(-2\mathfrak{A}+\mathfrak{r}A)(2kI_2+l\tilde{\mathfrak{A}}A)^{-1}\check{X}-\frac{l}{2}(-2\mathfrak{A}+\mathfrak{r}A)(2kI_2+l\tilde{\mathfrak{A}}A)^{-1}\tilde{\mathfrak{A}}\check{p}'+\frac{\mathfrak{r}}{2}\check{p}' \label{lag1}
\end{equation}
need to coincide with the relation
\begin{equation}
\check{Y}=\frac{a}{n}E_{11}\check{X}+\frac{1}{n}\check{p}''. \label{lag2}
\end{equation}
In particular, by comparing the slope of (\ref{lag1}) with the slope of (\ref{lag2}), although we obtain
\begin{equation*}
(-2\mathfrak{A}+\mathfrak{r}A)(2kI_2+l\tilde{\mathfrak{A}}A)^{-1}=\frac{a}{n}E_{11},
\end{equation*}
this contradicts the statement of Lemma \ref{lem}. Thus, indeed, the exact triangle (\ref{TR4}) gives a counterexample for Problem \ref{conj} under the assumption that the homological mirror symmetry conjecture for $(T^4_{J=T'}, \check{T}^4_{J=T'})$ holds true.

\section*{Acknowledgment}
I would like to thank Hiroshige Kajiura for various advices in writing this paper. I am also grateful to Satoshi Sugiyama for helpful comments. Finally, I would like to thank the referee for reading this paper carefully. This work was supported by Grant-in-Aid for JSPS Research Fellow 18J10909 and JSPS KAKENHI Grant Number JP16H06337.

\end{document}